\documentclass[12pt, a4article]{amsart}   	
\usepackage[top=1.17in, bottom=1.17in, left=1.19in, right=1.19in]{geometry}

\usepackage[colorlinks=true, pdfstartview=FitV, linkcolor=blue, citecolor=blue, urlcolor=blue, breaklinks=true]{hyperref}
\usepackage{amssymb,amscd,graphics, amsfonts, multicol}
\usepackage{graphics}
\usepackage{stmaryrd}
\usepackage{DotArrow}
\usepackage{amsmath, amsthm,wasysym}
\usepackage{epsf}
\usepackage{xypic}
\usepackage{tikz-cd}
\usepackage{color}
\usepackage[normalem]{ulem}
\theoremstyle{plain}

\newtheorem{theorem}{Theorem}[section]
\newtheorem{corollary}[theorem]{Corollary}
\newtheorem{proposition}[theorem]{Proposition}
\newtheorem{lemma}[theorem]{Lemma}
\newtheorem{thmx}{Theorem}

\theoremstyle{definition}

\newtheorem{remark}[theorem]{Remark}

\newtheorem{example}[theorem]{Example}
\newtheorem{definition}[theorem]{Definition}
\definecolor{fondo}{rgb}{0.898,0.996,0.898}

\title{Combinatorial Seshadri stratifications on normal toric varieties}
\author{Rocco Chiriv\`i}
\address{Dipartimento di Matematica e Fisica ``Ennio De Giorgi'', Universit\`a del Salento, Lecce, Italy}
\email{rocco.chirivi@unisalento.it}

\author{Martina Costa Cesari}
\address{Dipartimento di Matematica, Universit\`a di Bologna, Italy}
\email{martina.costacesari2@unibo.it}

\author{Xin Fang}
\address{Lehrstuhl f\"ur Algebra und Darstellungstheorie, RWTH Aachen, Pontdriesch 10-16, 52062 Aachen, Germany}
\email{xinfang.math@gmail.com}

\author{Peter Littelmann}
\address{Department Mathematik/Informatik, Universit\"at zu K\"oln, 50931, Cologne, Germany}
\email{peter.littelmann@math.uni-koeln.de}

\begin{document}
\maketitle
\begin{abstract}
We apply the theory of Seshadri stratifications to embedded toric varieties $X_P\subseteq \mathbb P(V)$ associated with a normal lattice polytope $P$. The approach presented here is purely combinatorial and 
completely independent of \cite{CFL}. In particular, we get a close connection between 
a certain class of triangulations of the polytope $P$, Seshadri stratifications of $X_P$ arising from torus orbit closures, and the associated degenerate semi-toric varieties.
In the last section we show that the approach here and the one in \cite{CFL} produce the same quasi-valuations
and hence the same degenerations of $X_P$.
\end{abstract}

\section*{Introduction}
\subsection*{Motivation}
One of the aims of the theory of Seshadri stratifications \cite{CFL} on embedded projective varieties $X\subseteq \mathbb P(V)$
is to use geometric data (subvarieties, vanishing order of functions) to construct a flat degeneration of $X$ into a union of toric varieties $X_0$.
The construction is motivated by the fact that, though
the degenerate variety $X_0$ is often more singular than $X$, its combinatorial structure makes it easier
to understand, and information about $X_0$ can often be ``lifted" to information on $X$. 

In this article, we focus on the case where $X$ is a toric variety.
Let $T\simeq \mathbb K^n$ be a torus with character lattice $M$, where $\mathbb K$ is an 
algebraically closed field of characteristic zero. Given a full dimensional normal lattice polytope 
$P\subseteq M_{\mathbb R}:=M\otimes_{\mathbb{Z}}\mathbb{R}$, we denote by $X_P\subseteq \mathbb P(V)$ the associated embedded toric variety (Section~\ref{sec:toric:variety}). The rich combinatorial structure of toric varieties makes it possible to present the theory
of Seshadri stratifications on toric varieties in a way
which uses merely the common combinatorial tools related to toric varieties: the polytope $P$, the weight monoid
$S$ of the homogeneous coordinate ring of $X_P$, and a certain class of triangulations of $P$. 
The exposition given here is completely independent  of \cite{CFL}.

\subsection*{Combinatorial Seshadri stratification} Let $A$ be the set of faces of $P$. For $\sigma\in A$ there is a unique $T$-orbits $O_\sigma\subseteq X_P$,
 denote by $X_\sigma$ its Zariski-closure in $X_P$.
For a collection of homogeneous $T$-eigenfunctions $f_\sigma\in\mathbb K[X_P]\setminus\{0\}$, indexed by $\sigma\in A$,
let $\mu_\sigma$ be the weight of $f_\sigma$ and set $m_\sigma=\deg f_\sigma$.
A \emph{combinatorial Seshadri stratification on $X_P$} is a collection of such pairs  $(X_\sigma,f_\sigma)_{\sigma\in A}$,
fulfilling the following compatibility condition: 
\begin{center}
for all $\sigma\in A$, the rational weight $-\frac{\mu_\sigma}{m_\sigma}$  lies in the relative interior of the face $\sigma$.
\end{center}

Given a combinatorial Seshadri stratification on $X_P$, for any fixed maximal chain $\mathfrak{C}$ in $A$, the set 
$\mathbb B_\mathfrak C=\{(m_\sigma,\mu_\sigma)\mid \sigma\in \mathfrak C\}$
turns out to be a basis of $\mathbb{Q}\oplus M_{\mathbb{Q}}$. Just using linear algebra,
we construct  for every maximal chain $\mathfrak C\subseteq A$
a valuation $\nu_\mathfrak C:\mathbb K[X_P]\setminus\{0\}\rightarrow \mathbb Q^\mathfrak C$
as follows:
for a homogeneous $T$-eigenfunction $f\in \mathbb K[X_P]$ of degree 
$m_f$ and weight $\mu_f$, the valuation $\nu_\mathfrak C(f)$ is  given by the coefficients of the expression 
of $(m_f,\mu_f)$ as a $\mathbb Q$-linear combination of the basis $\mathbb B_\mathfrak C$. 
The quasi-valuation $\nu:\mathbb K[X_P]\setminus\{0\}\rightarrow \mathbb Q^A$ 
associated to the combinatorial Seshadri stratification is defined by 
$$\nu(f)={\min}_{>^t}\{\nu_\mathfrak C(f)\mid \mathfrak C\ \textrm{maximal chain in\,}A\}\subseteq \mathbb Q^A,$$ 
where, in order to define the minimum, a linearization $>^t$ of the partial order on $A$ is fixed.
This quasi-valuation induces a filtration on $\mathbb K[X_P]$ by ideals, let $\mathrm{gr}_\nu\mathbb K[X_P]$ 
be the associated graded ring and set $X_0=\mathrm{Proj\,}(\mathrm{gr}_\nu\mathbb K[X])$.

This construction raises many natural questions: Are there many combinatorial Seshadri stratifications on $X_P$? 
What is the structure of $X_0$? How is the geometry of $X_P$ related to the
geometry of $X_0$? How to determine explicitly $\nu(f)$, etc... 

\subsection*{Results}
The aim of this article is to give answers to the above questions.
\subsubsection*{1}
First of all, to achieve a classification, we introduce on the set of combinatorial Seshadri stratifications on $X_P$ an equivalence relation (Definiton~\ref{equivalence:relation})
which ensures $\mathrm{gr}_\nu\mathbb K[X_P]\simeq \mathrm{gr}_{\nu'}\mathbb K[X_P]$
if $\nu$ and $\nu'$ are quasi-valuations associated to equivalent combinatorial Seshadri stratifications. 
Recall that a flag of faces in $A$ is chain, i.e. a totally ordered subset of $A$, 
see Section~\ref{Sec:Triangulations:and:Flags}.
A triangulation  $\mathcal T=(\Delta_C)_{C\in \mathcal F(A)}$ of $P$ indexed by flags of faces is, roughly
speaking, a marking $\{\mathbf v_\sigma\}_{\sigma\in A}$ of the faces of $P$ by rational points in the relative
interior of the faces. The simplices of the triangulation are given by: for a flag of faces $C$
the simplex $\Delta_C$ is the convex hull of points $\{\mathbf v_\sigma\}_{\sigma\in C}$.
\begin{thmx}
There is a bijection between the set of equivalence classes of combinatorial Seshadri stratifications
on $X_P$ and triangulations $\mathcal T$ of $P$ indexed by flags of faces in $A$.
\end{thmx}

\subsubsection*{2}
The quasi-valuation can be completely expressed in terms of the triangulation: let $f\in\mathbb K[X_P]\setminus\{0\}$ be a homogeneous
$T$-eigenfunction of degree $m_f$ and weight $\mu_f$, and let $\mathfrak C\subseteq A$ be a maximal chain.
We show that the following statements are equivalent:
\begin{itemize}
\item[(i)] $\nu(f)= \nu_\mathfrak C(f)$;
\item[(ii)] $-\frac{\mu_f}{m_f}\in\Delta_\mathfrak C$;
\item[(iii)] $\nu_\mathfrak C(f)$ has only non-negative entries.
\end{itemize}
Being coordinates with respect to a basis,  $\nu_{\mathfrak{C}}$ can be easily computed; and the quasi-valuation can be determined using the equivalence of (i) and (ii) above.

Moreover, if one of the equivalent conditions holds, then $\nu(f)$ is, up to rescaling,
given by the coefficients of the expression of $-\frac{\mu_f}{m_f}$ as affine linear combination
of the vertices of $\Delta_\mathfrak C$.

\subsubsection*{3} 
The next task is to describe the structure of $\mathrm{gr}_\nu\mathbb K[X_P]$.
\emph{A priori}, one has a dependence on the choice of  $``>^t$''. But, by the result above, 
it turns out that the relevant properties of $\nu$ only depend on the triangulation $\mathcal T$
and not on the choice of the linearization.

Let $S\subseteq \mathbb Z\oplus M$ 
be the weight monoid of the embedded toric variety. For a flag of faces $C\subseteq A$ let
$K(\Delta_C)\subseteq \mathbb R\oplus M_{\mathbb R}$ be the cone over the simplex
and set $S_C:=S\cap K(\Delta_C)$. The union of the cones $K(\Delta_C)$ defines a fan of cones,
where $C$ is running over all flags of faces in $A$. In the same way, the union of the 
$S_C$ defines a fan of monoids $S_\mathcal T$, where $C$ is running over all flags of faces in $A$.
We show in Section~\ref{sec:fan:algebra}:
\begin{thmx}
Denote by $\Gamma=\{\nu(f)\mid f\in\mathbb K[X_P] \setminus\{0\}\}\subseteq \mathbb Q^A$ the image
of the quasi-valuation $\nu$.
\begin{itemize}
\item[(i)] $\Gamma$ is a fan of monoids, isomorphic to
$S_{\mathcal T}$.  In particular, $\Gamma$ depends, up to isomorphism, only on the triangulation
$\mathcal T$ and is independent of the choice of the linearization $>^t$ of $A$.
\item[(ii)] 
The associated graded algebra $\mathrm{gr}_\nu\mathbb K[X_P]$ is isomorphic to the fan 
algebra $\mathbb K[\Gamma]$. In particular, the algebra $\mathrm{gr}_\nu\mathbb K[X_P]$ 
depends only on the triangulation $\mathcal T$.
\item[(iii)] 
The variety $X_0=\mathrm{Proj}\,(\mathrm{gr}_\nu\mathbb K[X_P])$ is 
reduced. It is the irredundant union of the toric varieties $\mathrm{Proj}\, (\mathbb K[S_\mathfrak C])$,
where $\mathfrak C$ runs over of all maximal chains in $A$. The variety $X_0$ is equidimensional,
i.e. all irreducible components of $X_0$ have same dimension as $X_P$.
\end{itemize}
\end{thmx}

\subsubsection*{4} 
The next step is to view $X_0$ as the special fibre in a flat family. Using the connection 
between the quasi-valuations $\nu$, monomial preorders and approximations by integral weight orders
we show (Section~\ref{homogenizatíon}):
\begin{thmx}
There exists a variety $\breve{\mathfrak X}_P$ together with a flat surjective morphism
$\pi:\breve{\mathfrak X}_P\rightarrow \mathbb A^1$ such that the fibre over $0\in\mathbb{A}^1$ is isomorphic to $X_0$, and
$\pi$ is trivial over $\mathbb{A}^1\setminus\{0\}$ with fibre isomorphic to $X_P$.
\end{thmx}
In Section~\ref{weight:function} we provide another way to look at the degeneration in the above theorem: Indeed, there exists an embedding of $X_P$ and $X_0$ into a weighted projective space
$\mathbb P(m_1,\ldots,m_r)$, endowed with the action of a one dimensional torus $\mathbb G_m$, such that 
one can view $X_0$ as a ``limit variety'': $\lim_{s\rightarrow 0} s\cdot X_P=X_0$.

In particular, if all vertices of the triangulation $\mathcal T$ are lattice points, then one can attach 
to each maximal chain $\mathfrak C\subseteq A$ a maximal simplex $\Delta_\mathfrak C$ 
and a toric variety $X_{\Delta_\mathfrak C} \subseteq \mathbb P(V)$.
The one dimensional torus $\mathbb G_m$ above acts also on the projective space $\mathbb P(V)$ we started
with, and it makes sense  to study the limit $\mathbb X_0:=\lim_{s\rightarrow 0} s\cdot X_P$ inside $\mathbb P(V)$, 
similar to what was done in \cite{Zhu}.
We prove the following result in Section~\ref{Zhu}, which can be thought of as a special case of the results in \emph{ibid.}: The limit $\mathbb X_0$ inside $\mathbb P(V)$
is the union of the toric varieties $X_{\Delta_\mathfrak C}$, where $\mathfrak C$ is running over all maximal chains 
in $A$.

The limit varieties $X_0 \subseteq \mathbb P(m_1,\ldots,m_r)$ and $\mathbb X_0\subseteq \mathbb P(V)$
are strongly connected. For example, the irreducible
component $X_\mathfrak C$ of $X_0\subseteq \mathbb P(m_1,\ldots,m_r)$, where $\mathfrak C\subseteq A$ a maximal chain,
is isomorphic to the normalization of $X_{\Delta_\mathfrak C}$, the corresponding irreducible component of $\mathbb X_0$.
The limit $\lim_{s\rightarrow 0} s\cdot I(X_P)$ of the vanishing ideal $I(X_P)$ of $X_P\subseteq \mathbb P(V)$
is in general not a radical ideal (see also \cite{Zhu}), for details about the  radical  
see Theorem~\ref{shadow:theorem}.

\subsubsection*{5}
In the last section we compare the notion of a combinatorial Seshadri stratification in this article
with the notion of a Seshadri stratification in \cite{CFL}. We show that: A combinatorial Seshadri stratification $(X_\sigma,f_\sigma)_{\sigma\in A}$ on $X_P$
is a Seshadri stratification on $X_P$ in the sense of \cite{CFL}, which is equivariant with respect to the $T$-action on $X_P$.
The quasi-valuation $\nu$ associated to a combinatorial Seshadri stratification $(X_\sigma,f_\sigma)_{\sigma\in A}$
is the same as the quasi-valuation $\mathcal V$ associated to the Seshadri stratification in \cite{CFL}. The associated semi-toric varieties are therefore the same.

\subsection*{Outlooks}
In the case of toric varieties, the role played by the subvarieties in the usual framework of a Seshadri stratification is completely 
replaced by the triangulations indexed by flags of faces. In a forthcoming article we will generalize this approach and consider,
as it was done in the rank one case in \cite{GKZ}, 
triangulations of the polytope $P$ as a natural starting point to construct higher rank quasi-valutions on the homogeneous coordinate ring
of an embedded normal toric variety, and to describe  the corresponding degenerate variety.

Another line of generalization we are approaching is to replace toric varieties by spherical varieties. 
It is a class of varieties which is endowed with a large collection of combinatorial tools, and we plan to use this
for a combinatorial description of the structure of the degenerate variety. But already the case of the flag variety
(see \cite{CFL3,CFL4}) shows that the transition from toric varieties to spherical varieties can not be done with ease.

\subsection*{Organization of the article}
The article is structured as follows. In Section~\ref{sec:toric:variety}
we recall a few standard facts about normal toric varieties and we fix some notation.
In Section~\ref{Sec:TSS} we introduce the notion of a combinatorial Seshadri 
stratification and in Section~\ref{Sec:Triangulations:and:Flags} we introduce
the triangulations of $P$ indexed by flags of faces. 
In Section~\ref{quasi-valuation:sec} we define 
the quasi-valuation associated to a  combinatorial Seshadri stratification on $X_P\subseteq \mathbb P(V)$
and prove a few first properties. In Section~\ref{sec:one:parameter} we discuss a flat degeneration of $X_P$ into $X_0$
given by a one parameter group, and in Section~\ref{Zhu} we discuss the  case where the vertices
of the triangulation are all lattice points.
In the last section, Section~\ref{Sec:SS}, we show that a combinatorial Seshadri stratification
is indeed a Seshadri stratification as defined in \cite{CFL}.

\vskip 10pt
\noindent
\textbf{Acknowledgments:} The work of R.C. is partially supported by PRIN 2022 S8SSW2 ``Algebraic and geometric aspects of Lie theory" (CUP I53D23002410006). The work of M.C.C. is founded by PRIN 2022 A7L229 ``ALgebraic and TOPological combinatorics" (CUP J53D23003660006). The work of X.F. is funded by the Deutsche
Forschungsgemeinschaft: “Symbolic Tools in
Mathematics and their Application” (TRR 195, project-ID 286237555). The work of P.L. is partially supported by DFG SFB/Transregio 191 ``Symplektische Strukturen in Geometrie, Algebra und Dynamik''.


\section{Polytopes and embedded normal toric varieties}\label{sec:toric:variety}
We fix some notation and recall a few standard facts about embedded normal toric varieties.
We denote by $\mathbb{K}$ an algebraically closed field of characteristic zero. 
Let $T\simeq (\mathbb K^*)^n$ be a torus with character lattice $M$ and dual lattice $N$.  We write $\langle\cdot ,\cdot \rangle$
for the non-degenerate pairing on $N\times M$ defined by: $\langle\eta,\mu\rangle$
is the unique integer such that $\mu(\eta(s))=s^{\langle\eta,\mu\rangle}$ for all $s\in\mathbb K^*$.

Let $P\subseteq M_\mathbb R=M\otimes_{\mathbb Z}\mathbb R$
be a full dimensional lattice polytope and set $\Lambda= P\cap M$.
Let $\{e_\chi\mid \chi\in \Lambda\}\subseteq \mathbb K^\Lambda$  be the standard basis of $V=\mathbb K^\Lambda$. 

\begin{definition}{(\cite[Chapter 2]{CLS})}\label{def:toric:variety}
The \textit{embedded toric variety $X_{P}\subseteq \mathbb P(V)$} is defined as the Zariski closure of the image of the  map 
$\iota: T \rightarrow \mathbb P(V),\quad t \mapsto \left[\sum_{\chi\in \Lambda} \chi(t) e_\chi\right]$.
\end{definition}
If $P$ is a normal polytope, then $X_P\subseteq \mathbb P(V)$ is a projectively normal variety.

For the rest of the article we assume that $P$ is a normal polytope.
\subsection{Orbits and faces}\label{sec:orbit:face}
We fix the coordinates on $V=\mathbb K^{\Lambda}$ and write $x_{\chi}$, $\chi\in \Lambda$, 
for the linear function on $V$ dual to $e_\chi$. 
Denote by $\hat X_{P}\subseteq V$ the affine cone over  $X_{P}$. 
Let $\hat T$ be the torus $\mathbb K^*\times T$ with character lattice $\hat M=\mathbb Z\oplus M$. Let  $\hat\iota:\hat T\rightarrow V$ be the map 
$(c,t)  \mapsto \sum_{\chi\in\Lambda} \chi(t) ce_\chi$, then $\hat X_{P}$ is the closure
of the image of $\hat\iota$ and $(x_{\chi}\circ \hat\iota)(c,t) =c\chi(t)$ for $(c,t)\in\hat T$.
One has a bijection between the $T$-orbits in $X_P$ and the faces of $P$ (\cite[Section 2.3, Section 3.2]{CLS}).
Let $\sigma$ be a face of $P$ and set $\Lambda_\sigma=\Lambda\cap\sigma$.
The $T$-orbit associated to a face $\sigma$ of $P$ is $O_\sigma= \{  [\sum_{\chi\in\Lambda_\sigma} \chi(t)e_\chi]\mid  t\in T\}$.
In particular, the coefficients of the $e_\chi$, $\chi\in \Lambda_\sigma$, are nonzero. 
\begin{definition}\label{defn:orbit:closure}
The Zariski closure $\overline{O_\sigma}\subseteq X_P$ of the orbit is denoted by
$X_\sigma$.
\end{definition}
The variety $X_\sigma$ is a toric variety associated to the polytope $\sigma$,
where we view the latter as a full dimensional lattice polytope in its affine span (see \cite[Section 3.2]{CLS}).

Let $A$ be the set of faces of $P$. This set is partially ordered by the inclusion relations on the faces: 
we write $\tau\le \sigma$ for 
$\sigma, \tau\in A$, if and only if $\tau\subseteq \sigma$. We have for the orbit closures: $X_\tau \subseteq X_\sigma$ 
if and only if $\tau\le \sigma$.

\subsection{The homogeneous coordinate ring of $X_P$}\label{homogeneous:coordinate}
The homogeneous coordinate ring $\mathbb K[X_{P}]$ of 
$X_{P}\subseteq \mathbb P(V)$ 
and the coordinate ring $\mathbb K[\hat X_{P}]$ of the affine cone $\hat X_{P}$ are the same rings.
Since we often work with $\hat X_P$, we will use for the rest of the article 
the notation $\mathbb K[\hat X_{P}]$. 
This ring has a natural grading 
$\mathbb K[\hat X_{P}]=\bigoplus_{m\ge 0}\mathbb K[\hat X_{P}]_m$. 

The action of $\hat T$ on $V$ induces a natural action
of $\hat T$ on $\mathbb K[V]$: for $\hat t=(c,t)\in \mathbb K^*\times T$ and $g\in \mathbb K[V]$ let 
$\hat t\cdot g$ be the regular function: $\hat X\rightarrow \mathbb K$, $x\mapsto g(t^{-1}.(c^{-1} x))$.
The coordinate functions $x_\chi$, $\chi\in \Lambda$, are $\hat T$-eigenfunctions of weight $(-1,-\chi)$.
Note that $\hat T$-eigenfunctions are automatically homogeneous, the degree is the absolute value
of the first entry of the $\hat T$-weight.
Since $\hat X_P$ is $\hat T$-stable, we get an induced $\hat T$-action on $\mathbb K[\hat X_{P}]$.

\begin{definition}
Let $Q$ be a rational polytope in $M_\mathbb R$. The  \textit{cone $K(Q)\subseteq \mathbb R\oplus M_\mathbb R$ associated to $Q$}
is the subset $K(Q):=\{(c,cp)\mid c\in \mathbb R_{\ge 0},p\in Q\}$.
\end{definition}

Let $K(P)\subseteq \mathbb R\oplus M_\mathbb R$ be the cone over $P$.
Denote by $S\subseteq K(P)$  the submonoid generated by the elements 
$(1,\chi)$, $\chi\in \Lambda$. The assumption that $P$ is a  full dimensional normal lattice polytope  
implies: $S = K(P)\cap (\mathbb Z\times M)$ (\cite[Lemma 2.2.14]{CLS}). The monoid $S$ is called the {\it weight monoid}
associated to $X_P\subseteq \mathbb P(V)$.

A monomial $\prod_{\chi\in\Lambda}x_{\chi}^{a_\chi}\in \mathbb K[V]$ is a $\hat T$-eigenvector, 
its $\hat T$-weight is $(-m,-\eta)$, where $(m,\eta)=\sum_{\chi\in\Lambda} a_\chi(1,\chi)$ is an element in
the weight monoid $S$. 
The algebra $\mathbb K[\hat X_{P}]$ is linearly spanned by the restrictions  
$\prod_{\chi\in\Lambda}x_{\chi}^{a_\chi}\vert_{\hat X_{P}}$ of these monomials.

For every $(m,\eta)\in S$, we fix a decomposition $(m,\eta)= \sum_{\chi\in\Lambda} a_\chi(1,\chi)$,
and set $f_{m,\eta}=\prod_{\chi\in\Lambda}x_{\chi}^{a_\chi}\vert_{\hat X_{P}}$.
The following is well known:
\begin{lemma}\label{explicit:basis}
The function $f_{m,\eta}$ depends only on $(m,\eta)$ and not on the choice of the 
decomposition of $(m,\eta)$. The set $\{f_{m,\eta}\mid (m,\eta)\in S \}$ is a 
$\mathbb K$-basis for $\mathbb K[\hat X_{P}]$. 
\end{lemma}

The subspaces of homogeneous  functions are given by 
$\mathbb K[\hat X_P]_m=\bigoplus_{\chi\in (mP\cap M)}\mathbb K f_{m,\chi}$, $m\in\mathbb N$ (see \cite[Example~4.3.7]{CLS}).

\section{Combinatorial Seshadri stratifications and valuations}\label{Sec:TSS}
Let $X_P\subseteq \mathbb P(V)$ be an embedded toric variety as in Section~\ref{sec:toric:variety}
and recall that $A$ denotes the set of faces of the polytope $P$.
Let $X_\sigma$, $\sigma\in A$, be the collection of $T$-orbit closures in $X_P$ 
and let  $f_\sigma \in \mathbb K[\hat X_P]$, $\sigma\in A$, be a collection of homogeneous $T$-eigenfunctions 
of degree $\deg f_\sigma\ge 1$. Denote by $\mu_\sigma=(-\deg f_\sigma,\tilde\mu)$, $\tilde\mu\in M$, the $\hat T$-weight of $f_\sigma$.
In the following we identify $P\subseteq M_\mathbb R$ with the polytope in
the affine hyperplane $\{(1,\eta)\mid \eta\in M_\mathbb R\}\subset\hat M_\mathbb R$
obtained as the convex hull of the points $(1,\chi)$, $\chi\in\Lambda$.

\begin{definition}\label{defn:T_equi:SS}
The collection $(X_\sigma,f_\sigma)_{\sigma\in A}$ of $T$-orbits closures $X_\sigma\subseteq X_P$ and $\hat T$-eigen\-functions
$f_\sigma\in \mathbb K[\hat X_P]$ is called a  
\emph{combinatorial Seshadri stratification on $X_P$} if the $\hat T$-weights $\mu_\sigma$ of the functions $f_\sigma$, $\sigma\in A$,
satisfy the following condition:
\begin{equation}\label{weight:condition11}
\forall \sigma\in A, \ \frac{-\mu_\sigma}{\deg f_\sigma}\textrm{\it\ is a point in the relative interior $\sigma^o$ of $\sigma$}.
\end{equation}
\end{definition}

\begin{example}\label{square1x1}
A natural choice of the functions $f_\sigma$ is the product of the linear functions associated to vertices of $\sigma$.

	Let $T=(\mathbb K^*)^2$ be the torus. The character lattice of $T$ is $\mathbb  Z e_1\oplus\mathbb Z e_2$ where $\{e_1,e_2\}$ is the natural basis of $\mathbb R^2$. Let $P$ be the square with set of  vertices $E=\{v_0=(0,0), \ v_1=(1,0), \  v_2=(0,1), \ v_3=(1,1)\}$. Then $\Lambda=E$.
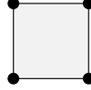
	$$
	\begin{tikzpicture}
		\fill[lightgray!20] (0,0) rectangle  (1,1);
			\filldraw  [black] (0,0) circle (2pt);
		\filldraw [black] (0,1) circle (2pt);
		\filldraw  [black] (1,0) circle (2pt);
		\filldraw [black] (1,1) circle (2pt);
		\draw (0,0) rectangle (1,1);
		\end{tikzpicture}
		$$
	The faces of $P$ are: the vertices in $E$, the $4$ edges $e_{(i,j)}$ connecting $v_i$ with $v_j$ where $(i,j)\in\{(0,1), \ (0,2), \ (1,3), \ (2,3))\}$, and the whole square $P$. The homogeneous $\hat{T}$-eigenfunctions
	$$
	f_{v_i}=x_{v_i}, \ f_{e_{(i,j)}}=x_{v_i}x_{v_j}, \  f_{P}=x_{v_0}x_{v_1}x_{v_2}x_{v_3}
	$$
	satisfy the property (\ref{weight:condition11}).
	\end{example}

\begin{example}\label{integral:case}
Suppose $P$ has enough lattice points in the following sense: for every face $\sigma\in A$ there
exists a weight $\chi_\sigma\in \Lambda\cap \sigma^o$ in the relative interior; fix a collection of such weights $\chi_\sigma$, $\sigma\in A$.
The linear functions $f_\sigma=x_{\chi_\sigma}$, $\sigma\in A$,
satisfy the condition on the weights, so the collection 
$(X_\sigma,f_\sigma)_{\sigma\in A}$ of subvarieties and $\hat T$-eigenfunctions is a combinatorial Seshadri stratification.
\end{example}

\begin{example}\label{square2x2}
	Let $T$ and $M$ be as in Example~\ref{square1x1}.
		 Let $P$ the square with set of  vertices $E=\{v_0=(0,0), \ v_1=(2,0), \  v_2=(0,2), \ v_3=(2,2)\}$. Then $\Lambda=E\cup \{ (1,0), \  (0,1), \ (1,1), \ (2,1), \ (1,2) \}$.
		$$
		\begin{tikzpicture}
			\fill[lightgray!20] (0,0) rectangle  (2,2);
			\filldraw  [black] (0,0) circle (2pt);
			\filldraw [black] (0,1) circle (2pt);
			\filldraw  [black] (1,0) circle (2pt);
				\filldraw [black] (2,1) circle (2pt);
			\filldraw  [black] (1,2) circle (2pt);
			\filldraw [black] (1,1) circle (2pt);
			\filldraw [black] (0,2) circle (2pt);
			\filldraw  [black] (2,0) circle (2pt);
			\filldraw [black] (2,2) circle (2pt);
			\draw (0,0) rectangle (2,2);
		\end{tikzpicture}
		$$
		The polytope $P$ has enough lattice points. So the functions $f_{\sigma}$ corresponding to the coordinate associated with the weights in $\Lambda\cap \sigma$, satisfy the property (\ref{weight:condition11}).
	\end{example}

By a \textit{marking $\mathfrak{m}$ of the faces} of the polytope $P$ we mean a collection $\mathfrak{m}=(u_\sigma)_{\sigma\in A}$ 
of points such that $u_\sigma$ is a rational point in the relative interior $\sigma^o$ of the face 
$\sigma\subseteq P$. The collection of weights in \eqref{weight:condition11} defines such a marking of the faces.

\begin{definition}\label{equivalence:relation}
Let $(X_\sigma,f_\sigma)_{\sigma\in A}$ be a  combinatorial Seshadri stratification on $X_P$
and denote by $\mu_\sigma$ the $\hat T$-weight of the function $f_\sigma$, $\sigma\in A$. We call
$\mathfrak{m}_{\mathbf f}=(\frac{-\mu_\sigma}{\deg f_\sigma})_{\sigma\in A}$ the \textit{associated marking of the faces}.
Two combinatorial Seshadri stratifications $(X_\sigma, f_\sigma)_{\sigma\in A}$ and 
$(X_\sigma, h_\sigma)_{\sigma\in A}$ are called \textit{equivalent}
if the associated markings are equal: $\mathfrak{m}_{\mathbf f}=\mathfrak{m}_{\mathbf h}$.
\end{definition}

\begin{remark}\label{T:equivariant:Compatible:subvariety}
Let $(X_\sigma,f_\sigma)_{\sigma\in A}$ be a combinatorial Seshadri stratification on $X_P\subseteq \mathbb P(V)$.
For $\sigma\in A$ let $M_{\mathbb R,\sigma}\subseteq M_{\mathbb R}$ be the affine span of $\sigma$.
Set $M_\sigma=M\cap M_{\mathbb R,\sigma}$. So $\sigma$ is a full dimensional lattice polytope in $M_{\mathbb R,\sigma}$; set $V_\sigma=\langle \hat X_\sigma\rangle_\mathbb K\subseteq V$ and
$A_\sigma=\{\kappa\in A\mid\kappa\le\sigma\}$. The collection of varieties $X_\tau\subseteq X_\sigma$
and functions $f_\tau\vert_{X_\sigma}$, $\tau\in A_\sigma$, defines a combinatorial Seshadri stratification for the normal embedded toric variety 
$X_\sigma\hookrightarrow  \mathbb P(V_\sigma)\subseteq\mathbb P(V)$.
\end{remark}
\subsection{First properties}
Let $(X_\sigma,f_\sigma)_{\sigma\in A}$ be a combinatorial Seshadri stratification on $X_P$ and denote by $\mu_\sigma$
the $\hat T$-weights of the extremal functions $f_\sigma$, $\sigma\in A$.
\begin{lemma}\label{vanishing:set}
For $\sigma,\tau\in A$, the restriction $f_\sigma\vert_{\hat X_\tau}\equiv 0$ is  identically zero if 
and only if $\tau\not\ge \sigma$.
\end{lemma}
\begin{proof}
Since $\frac{-\mu_\sigma}{\deg f_\sigma}\in \sigma^o$ is an element in the relative interior of $\sigma$,
it follows that $f_\sigma$ can be written as the restriction of a monomial in the $x_\chi$, $\chi\in \Lambda_\sigma$.
All the coefficients of the $e_\chi$ in the expression $y=[\sum_{\chi\in \Lambda_\tau} a_\chi e_\chi]\in O_\tau$
are nonzero (Section~\ref{sec:orbit:face}), so the restriction of the function to $X_\tau$ is not identically zero if $\tau\ge \sigma$
because $\Lambda_\sigma\subseteq \Lambda_\tau$.

If $\tau\not\ge \sigma$, then $\sigma$ is not a face of $\tau$. Since $\frac{-\mu_\sigma}{\deg f_\sigma}\in \sigma^o$, 
there exists at least one $\chi\in\Lambda_\sigma \setminus \Lambda_\tau$ such that $x_\chi$ is a factor of $f_\sigma$,
and hence the restriction of $f_\sigma$ to $O_\tau$ is identically zero.
\end{proof}
\begin{lemma}\label{marking:bijection}
The map $\Psi$, which associates to a combinatorial Seshadri stratification $(X_\sigma,f_\sigma)_{\sigma\in A}$
the marking $\mathfrak{m}_{\mathbf f}=(\frac{-\mu_\sigma}{\deg f_\sigma})_{\sigma\in A}$,
induces a bijection between the set of equivalence classes of  combinatorial Seshadri stratifications on $X_P$  and the set 
of markings of the faces of $P$.
In every equivalence class of combinatorial Seshadri stratifications there exists an element $(X_\sigma,f_\sigma)_{\sigma\in A}$
such that all other combinatorial Seshadri stratifications in the same equivalence class are equal to 
$(X_\sigma,c_\sigma f^{\ell_\sigma}_\sigma)_{\sigma\in A}$ for some $c_\sigma\in\mathbb K^*$, $\ell_\sigma\in\mathbb N_{>0}$,
$\sigma\in A$.
\end{lemma}
\begin{proof}
The map $\Psi$ is well defined and injective on the equivalence classes.  Let $\mathfrak m=(u_\sigma)_{\sigma\in A}$ 
be a marking of the faces. For $\sigma\in A$,
the set $\{n\in \mathbb Z\mid n u_\sigma\in \hat M\}\subseteq \mathbb Z$ is an ideal,
let $q_\sigma>0$ be a generator of the ideal. For $\sigma\in A$ set $\lambda_\sigma=-q_{\sigma}u_\sigma\in \hat M$,
this is the weight of a $\hat T$-eigenfunction $f_\sigma$ of degree $q_\sigma$, and $(X_\sigma,f_\sigma)_{\sigma\in A}$ 
is a combinatorial Seshadri stratification with associated marking $\mathfrak m_{\mathbf f}=(u_\sigma)_{\sigma\in A}$.
It follows that $\Psi$ is a bijection, we fix these functions $f_\sigma$, $\sigma\in A$, for the rest of the proof.

A $\hat T$-eigenfunction $h_\sigma$ of weight $\nu_\sigma$ with $\frac{-\nu_\sigma}{\deg h_\sigma}=u_\sigma$ is equal to 
 $h_\sigma=c_\sigma f^{\ell_\sigma}_\sigma$ for some  $c_\sigma\in\mathbb K^*$, $\ell_\sigma\in\mathbb N_{>0}$. 
It follows: the equivalence class of $(X_\sigma,f_\sigma)_{\sigma\in A}$ consists of 
combinatorial Seshadri stratifications of the form $(X_\sigma, c_\sigma f^{\ell_\sigma}_\sigma)_{\sigma\in A}$, 
$c_\sigma\in\mathbb K^*$, $\ell_\sigma\in\mathbb N_{>0}$,
$\sigma\in A$.
\end{proof}
\begin{lemma}\label{weights:lin:indep}
If $\mathfrak{m}=(u_\sigma)_{\sigma\in A}$ is a marking of the faces of $P$, then, for every maximal chain $\mathfrak C\subseteq A$,
the  rational weights $\{u_\sigma\}_{\sigma\in\mathfrak C}$ form a $\mathbb Q$-basis for $\hat M_{\mathbb Q}$.
\end{lemma}

\begin{proof}
By assumption, $u_\sigma$ is an element of the relative interior of $\sigma$. So
$u_{\sigma_1}-u_{\sigma_0}$, $u_{\sigma_2}-u_{\sigma_1}$, $\ldots$, $u_{\sigma_r}-u_{\sigma_{r-1}}$
are linearly independent and hence form a basis for $M_{\mathbb Q}$, which implies: $\{u_\sigma\}_{\sigma\in\mathfrak C}$
is a basis for $\hat M_{\mathbb Q}$.
\end{proof}

\subsection{Generalities on valuations}

\begin{definition}
A {\it quasi-valuation} on $\mathbb K[\hat X_P]$ with 
values in a totally ordered abelian group $G$  is a map $\nu: \mathbb K[\hat X_P]\setminus \{0\}\rightarrow \textsc{G}$ 
which 

\begin{itemize} 
\item[(i)] has the minimum property, i.e., $\nu(f + g) \ge \min\{\nu(f), \nu(g)\}$;
\item[(ii)] is not affected by scalar multiplication: $\nu(c f) = \nu(f)$ for all $c\in \mathbb K^*$;
\item[(iii)] is quasi-additive, i.e. $\nu(fg) \ge \nu(f) + \nu(g)$.
\end{itemize}
\noindent
We assume in all cases $f,g \in \mathbb K[\hat X_P]\setminus \{0\}$, and if appropriate, $f+g\not=0$.
The quasi-valuation $\nu$ is called \emph{homogeneous} if  $\nu(f^p) =p \nu(f)$ for all $p\in \mathbb N$ and $f\in\mathbb K[\hat X_P]\setminus \{0\}$.
The quasi-valuation $\nu$ is called a \emph{valuation} if  it is additive, i.e., $\nu(fg)= \nu(f) + \nu(g)$ for all 
$f,g \in \mathbb K[\hat X_P]\setminus \{0\}$.
\end{definition}

Given a quasi-valuation $\nu$ and $\underline a\in G$, the subset 
$$\mathbb K[\hat X_P]_{\ge \underline a}=\{h\in \mathbb K[\hat X_P]\setminus\{0\}\mid \nu(h)\ge \underline{a}\}$$
is an ideal.
The ideal $ \mathbb K[\hat X_P]_{> \underline a}$ is defined similarly. The quotient  $F_{\underline a}=
\mathbb K[\hat X_P]_{\ge \underline a}/\mathbb K[\hat X_P]_{> \underline a}$
is called a \textit{leaf} of the quasi-valuation. The direct sum of the leaves: 
$\bigoplus_{\underline{a}\in G} F_{\underline a}$ inherits an algebra structure and is
called the \textit{associated graded algebra}, denote it by $\mathrm{gr}_\nu \mathbb K[\hat X_P]$.

Let $(X_\sigma,f_\sigma)_{\sigma\in A}$ be a combinatorial Seshadri stratification and
fix a maximal chain $\mathfrak C=\{\sigma_r>\ldots>\sigma_0\}$ in $A$. Let $\mu_r,\ldots,\mu_0$
be the weights of the extremal functions $f_{\sigma_r},\ldots, f_{\sigma_0}$. 
By Lemma~\ref{weights:lin:indep},
these weights form a $\mathbb Q$-basis for $\hat M_{\mathbb Q}$.
Let  $\{e_{\sigma_r},\ldots, e_{\sigma_0}\}$ be the standard basis
of $\mathbb Q^\mathfrak C$. We endow $\mathbb Q^\mathfrak C$ with the lexicographic order, i.e.
$(a_r,\ldots,a_0)>(b_r,\ldots,b_0)$ if $a_r>b_r$ or $a_r=b_r$ and $a_{r-1}>b_{r-1}$ and so on.
In this way $\mathbb Q^\mathfrak C$ becomes a totally ordered abelian group.

\begin{definition}\label{def:valuation}
For a $\hat T$-eigenfunction $g\in  \mathbb K[\hat X_P]$ of weight $\lambda_g$, let $\lambda_g=a_r\mu_r+\ldots +a_0\mu_0$
be the uniquely determined expression of $\lambda$ as a linear combination in terms of the basis given by the weights 
$\{\mu_r,\ldots,\mu_0\}$. We set $ \nu_{\mathfrak C}(g)=a_r e_{\sigma_r} + \ldots + a_0 e_{\sigma_0}$. For an arbitrary function 
 $g\in \mathbb K[\hat X_P]$, let $g=g_{\eta_1}+\ldots + g_{\eta_t}$ be a decomposition of $g$ into  $\hat T$-eigenfunctions
 of pairwise distinct weights $\eta_i$, $i=1,\ldots, t$. We define a map $\nu_{\mathfrak{C}}:\mathbb K[\hat X_P]\rightarrow \mathbb Q^\mathfrak C$ by:
$$
 \nu_{\mathfrak C}(g)=\min\{\nu_{\mathfrak C}(g_{\eta_i})\mid 1\leq i\leq t\}.
$$
\end{definition}
\begin{example}\label{extremal:valuation}
If $\sigma\in\mathfrak C$, then $\nu_{\mathfrak C}(f_\sigma)=e_\sigma$.
\end{example}
\begin{remark}\label{total:order:on:S}
The $\hat T$-weight spaces are one dimensional. If $g$ is a $\hat T$-eigenfunction,
then the value $\nu_{\mathfrak C}(g)$ on a $\hat T$-eigenvector depends only on the $\hat T$-weight
of the function, and vice versa, the weight of $g$ can be reconstructed from $\nu_{\mathfrak C}(g)$. 
\end{remark}

\begin{lemma}
The map $\nu_{\mathfrak C}:\mathbb K[\hat X_P]\setminus\{0\}\rightarrow \mathbb Q^\mathfrak C$ is a valuation with at most one-dimensional leaves.
\end{lemma}

\begin{proof}
The map $ \nu_{\mathfrak C}$ has the minimum property by construction, and it is not affected by non-zero scalar multiplication.
If $g=g_{\eta_0}+\ldots + g_{\eta_t}$ and $h=h_{\omega_0}+\ldots + h_{\omega_s}$ are decompositions 
of $g$ and $h$ into pairwise distinct $\hat T$-eigenfunctions of weight $\eta_i$, $i=0,\ldots, t$, respectively $\omega_j$, $j=0,\ldots s$,
such that $\nu_\mathfrak C(g)=\nu_\mathfrak C(g_{\eta_0})$ and 
$\nu_\mathfrak C(h)=\nu_\mathfrak C(h_{\omega_0})$,
then $gh$ is a sum of $\hat T$-eigenfunctions: $g_{\eta_i}h_{\nu_j}$, $i=0,\ldots, t$, $j=0,\ldots s$,
and $\nu_\mathfrak C(g_{\eta_i}h_{\omega_j})=\nu_\mathfrak C(g_{\eta_i})+\nu_\mathfrak C(h_{\omega_j})$. It follows:
$\nu_\mathfrak C(gh)=\nu_\mathfrak C(g_{\eta_0})+\nu_\mathfrak C(h_{\omega_0})=\nu_\mathfrak C(g)+\nu_\mathfrak C(h)$, so the map is additive and hence a valuation. Positive dimensional leaves are just $\hat T$-weight spaces by 
Remark~\ref{total:order:on:S}, so they are one-dimensional.
\end{proof}

We extend the valuation to $\nu_\mathfrak C: \mathbb K(\hat X)\setminus \{0\}\rightarrow \mathbb Q^\mathfrak C$
by $\nu_\mathfrak C(\frac{g}{h})=\nu_\mathfrak C(g)-\nu_\mathfrak C(h)$.
Denote by $F_\mathfrak C$ the product $\prod_{\sigma\in\mathfrak C} f_\sigma$ and let 
$\mathbb K[\hat X_P]_{F_\mathfrak C}\subseteq \mathbb K(\hat X)$ be the 
localization of the homogeneous coordinate ring. It is the coordinate ring of 
$U_{F_\mathfrak C}=\{x\in\hat X_P\mid F_\mathfrak C(x)\not=0\}$,
an affine toric $\hat T$-variety. Note that the valuation may have negative entries.
\begin{proposition}\label{prop:filtration} 
Let $g\in\mathbb K[\hat X_P]\setminus\{0\}$.
For $\nu_{\mathfrak C}(g)=(a_\sigma)_{\sigma\in\mathfrak C}$ let $k>0$ be a positive integer such that $ka_\sigma\in\mathbb Z$, $\sigma\in\mathfrak C$.
In $\mathbb K[\hat X_P]_{F_\mathfrak C}$ one has: $g^k=c\prod_{\sigma\in\mathfrak C} f^{ka_\sigma}_{\sigma}+h$, where $c\in\mathbb K^*$
and either $h=0$, or $\nu_{\mathfrak C}(h)> (ka_\sigma)_{\sigma\in\mathfrak C}$. 
\end{proposition}
\begin{proof}
Let $g=g_{\eta_0}+\ldots + g_{\eta_t}$ be a decomposition of $g$ into $\hat T$-eigenfunctions of pairwise different weights $\eta_i$, $i=0,\ldots, t$.
Without loss of generality we assume $\nu_{\mathfrak C}(g)=\nu_{\mathfrak C}(g_{\eta_0})$. It follows that 
$\nu_{\mathfrak C}(g^k)=k\nu_{\mathfrak C}(g_{\eta_0})$. This implies that $g_{\eta_0}^k$ and 
$\prod_{\sigma\in\mathfrak C} f_\sigma^{ka_\sigma}$ are $\hat T$-eigenfunctions with the same valuation and hence the same $\hat T$-weight.
So in $\mathbb K[\hat X_P]_{F_\mathfrak C}$ they must be nonzero scalar multiples of each other. The value of $\nu_\mathfrak C$ 
on the remaining summands in $g^k=(\sum_{i=0}^t g_{\eta_i})^k$ is strictly larger than $\nu_{\mathfrak C}(g^k_{\eta_0})$, which proves the claim.
\end{proof}
As an immediate consequence of Proposition~\ref{prop:filtration} we see:
\begin{corollary}\label{weight:and:degree:formula1}
Let $g\in \mathbb K[\hat X_P]\setminus\{0\}$ be a $\hat T$-eigenfunction of $\hat T$-weight $\lambda_g$ and 
suppose $\nu_{\mathfrak C}(g)=(a_\sigma)_{\sigma\in\mathfrak C}$. If $k$ is as in Proposition~\ref{prop:filtration}, 
then $g^k=c\prod_{\sigma\in\mathfrak C} f^{ka_\sigma}_{\sigma}$ in
$\mathbb K[\hat X_P]_{F_\mathfrak C}$ for some $c\in\mathbb K^*$, 
$\deg g=\sum_{\sigma\in\mathfrak C} a_\sigma\deg f_{\sigma}$ and $\lambda_g=\sum_{\sigma\in\mathfrak C} 
a_\sigma\mu_\sigma$.
\end{corollary}
We get also some information about the value of  $\nu_{\mathfrak C}(f_\tau)$ for $\tau\not\in\mathfrak C$:
\begin{corollary}\label{weight:and:degree:formula2}
For $\tau\in A$, $\tau\not\in\mathfrak C$, let $\nu_{\mathfrak C}(f_\tau)=(a_{\sigma_i})_{0\le i\le r}$. 
If $\sigma_j \in\mathfrak C$ is minimal such that
$\sigma_j>\tau$, then $a_{\sigma_j}>0$ and $a_{\sigma_i}=0$ for $i>j$.
\end{corollary}
\begin{proof}
Let  $\mathfrak C=(\sigma_r,\ldots,\sigma_0)$. By assumption, we have $\frac{-\mu_\tau}{\deg f_\tau}\in\tau^o\subseteq \sigma_j$,
which implies:  $\frac{-\mu_\tau}{\deg f_\tau}$ is in the linear $\mathbb Q$-span of the $\{\frac{\mu_{\sigma_i}}{\deg f_{\sigma_i}}\mid i=0,\ldots,j\}$, 
and hence $a_{\sigma_i}=0$ for $i>j$.
If $a_{\sigma_j}\le 0$, then let $k>0$ be as in Corollary~\ref{weight:and:degree:formula1} and set $N\gg 0$. By \emph{ibid.},
the regular functions $f_\tau^k f_{\sigma_j}^{\vert ka_{\sigma_j}\vert}(\prod_{i=0,\ldots,j-1} f_{\sigma_i})^N$ and $\prod_{i=0,\ldots,j-1} f^{ka_{\sigma_i}+N}_{\sigma_i}$ 
have the same $\hat T$-weight, so they are equal in $\mathbb K[\hat X_P]$ up to some nonzero
scalar multiple. But this is not possible: the first function vanishes identically
on $X_{\sigma_{j-1}}$ because $\sigma_{j-1}\not\ge \tau$ (Lemma~\ref{vanishing:set}),
but the second not. It follows: $a_{\sigma_j}>0$.
\end{proof}
\subsection{The valuation monoid}
The additivity of $\nu_{\mathfrak C}$ implies that 
$$\mathbb V_\mathfrak C(\hat{X}_P)=\{ \nu_{\mathfrak C}(h)\mid h\in  \mathbb K[\hat X_P]\setminus\{0\}\}\subseteq \mathbb Q^{\mathfrak C}$$
is a submonoid of  $\mathbb Q^{\mathfrak C}$, called the \emph{valuation monoid}.
For an arbitrary embedded projective variety $Y\subseteq \mathbb P(V)$ a valuation monoid
$\mathbb V(\hat{Y})$ is typically used to 
construct (if possible) a flat degeneration of $Y$ into a toric variety, having the valuation monoid as weight monoid.
So it is no surprise that in the case where $Y=X_P$ is a toric variety we do not get anything really new.  By 
Lemma~\ref{explicit:basis} and  Corollary~\ref{weight:and:degree:formula1} we see:
\begin{corollary}\label{weight:versus:valuation}
Let $\mu_r,\ldots,\mu_0$
be the weights of the extremal functions $f_{\sigma_r},\ldots, f_{\sigma_0}$.
The map $\psi_{\mathfrak C}: \mathbb V_{\mathfrak C}(\hat X_P)\rightarrow S$,
$(a_r,\ldots,a_0)\mapsto \sum_{j=0}^r a_j\mu_j$, is an isomorphism of monoids. 
\end{corollary}
A new point of view comes in Section~\ref{quasi-valuation:sec}. Let us just remark that for every maximal chain
we have a special submonoid: $\mathbb V_{\mathfrak C}(\hat X_P)^+=\mathbb V_{\mathfrak C}(\hat X_P)\cap \mathbb Q^\mathfrak C_{\ge 0}$,
the intersection of the valuation monoid with the positive orthant. We will see that the images 
$\psi_{\mathfrak C}(\mathbb V_{\mathfrak C}(\hat X_P)^+)\subseteq S$
define a decomposition of $S$ as $\mathfrak C$ is running over all maximal chains in $A$. 
 
    \begin{example}
    Let $X_P$ the toric variety with the combinatorial Seshadri stratification of Example \ref{square1x1}. Then, using Corollary \ref{weight:versus:valuation}, is easy to see that, for every maximal chain $\mathfrak C$, the valuation monoid $\mathbb V_{\mathfrak C}(\hat{X}_P)$ is the monoid
    $$
    \left\{ (a_2,a_1,a_0)\in \mathbb Q^3 \mid  
	2a_2\in\mathbb Z, \
	a_1\in\mathbb Z,\ 
	a_0\in\mathbb Z
     \right\}.
    $$
	\end{example}

	\begin{example}
		Let $X_P$ the toric variety with the combinatorial Sesahdri stratification of Example \ref{square2x2}. By Corollary \ref{weight:versus:valuation} one can see that, for every maximal chain $\mathfrak C$, the monoid  $\mathbb V_{\mathfrak C}(\hat X_P)$ is $\mathbb Z^3$.
	\end{example}

\section{Markings and triangulations of $P$ indexed by flags}\label{Sec:Triangulations:and:Flags}
A flag of faces in $A$ is a chain in $A$, i.e. a totally ordered subset of the form
 $\sigma_1\subsetneq \ldots\subsetneq \sigma_s$, where the $\sigma_i\in A$, $i=1,\ldots,s$, 
 are faces of $P$. Let $\mathcal F(A)$ be the set of all flags in $A$, i.e. the set of all totally ordered
subsets of $A$.
\begin{definition}\label{triangulation:by:flags1}
A \textit{triangulation of $P$ indexed by flags of faces} is a triangulation  $\mathcal T=(\Delta_C)_{C\in \mathcal F(A)}$
of $P$ with rational vertices and simplices $\Delta_C$ indexed by $\mathcal F(A)$, such that
\begin{enumerate}
\item the relative interior of every face $\sigma\in A$ contains exactly one vertex $\mathbf v_\sigma\in M_{\mathbb Q}$ of $\mathcal T$,
\item $\Delta_C$ is the convex hull of the $\mathbf v_\sigma$, $\sigma\in C$.
\end{enumerate}
\end{definition}
\begin{example}
The barycentric subdivision of $P$ is a triangulation of $P$ indexed by flags.
\end{example}

Given a triangulation indexed by flags $\mathcal T=(\Delta_C)_{C\in \mathcal F(A)}$,
the collection $(\mathbf v_\sigma)_{\sigma\in A}$ of vertices defines  a marking $\mathfrak{m}_{\mathcal T}$ of the faces of $P$.
Vice versa, let $\mathfrak{m}=(u_\sigma)_{\sigma\in A}$ be a marking of the faces. For a flag $C\in \mathcal F(A)$ let
$\Delta_C\subseteq P$ be the convex hull of the $u_\sigma$, $\sigma\in C$. 
By Lemma~\ref{weights:lin:indep}, the $u_\sigma$, $\sigma\in C$, are linearly independent and hence $\Delta_C$ is a 
simplex of dimension $\vert C\vert-1$. We attach to a marking of the faces $\mathfrak{m}=(u_\sigma)_{\sigma\in A}$ 
the collection of simplices  $\mathcal T_\mathfrak{m}=(\Delta_C)_{C\in \mathcal F(A)}$. 

\begin{lemma}
The collection of simplices $\mathcal T_\mathfrak{m}=(\Delta_C)_{C\in \mathcal F(A)}$ is a triangulation of $P$.
\end{lemma}
\begin{proof}
This is evident for $r=0$. The general case follows by induction on the dimension.

For each facet of $P$, we have a marking that, by induction, gives a triangulation of the facet. The collection of simplices $\mathcal T_m$ is the triangulation given by the cones with vertex $u_P$ over these triangulations of the facets.


\end{proof}

\begin{example}
Let $P$ be the polytope as in Example \ref{square2x2}. 
Let $\mathfrak m$ be the marking associated with the extremal functions given in Example \ref{square2x2}.
The triangulation associated with $\mathfrak m$ is
$$
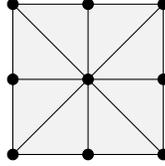
\begin{tikzpicture}
    \fill[lightgray!20] (0,0) rectangle  (2,2);
			\filldraw  [black] (0,0) circle (2pt);
			\filldraw [black] (0,1) circle (2pt);
			\filldraw  [black] (1,0) circle (2pt);
				\filldraw [black] (2,1) circle (2pt);
			\filldraw  [black] (1,2) circle (2pt);
			\filldraw [black] (1,1) circle (2pt);
			\filldraw [black] (0,2) circle (2pt);
			\filldraw  [black] (2,0) circle (2pt);
			\filldraw [black] (2,2) circle (2pt);
			\draw (0,0) rectangle (2,2);
                \draw (0,0) -- (1,1);
                \draw (1,0) -- (1,1);
                \draw (0,1) -- (1,1);
                \draw (2,0) -- (1,1);
                \draw (2,1) -- (1,1);
                \draw (0,2) -- (1,1);
                \draw (1,2) -- (1,1);
                \draw (2,2) -- (1,1);
\end{tikzpicture}
$$
\end{example}

Summarizing we get together with Lemma~\ref{marking:bijection}:
\begin{theorem}\label{SeshStrat:triang:marking}
There exists a bijection between the set of equivalence classes of combinatorial Seshadri stratifications
and the set of triangulations  of $P$ indexed by flags of faces.
\end{theorem}

\section{A higher rank quasi-valuation}\label{quasi-valuation:sec}
Let $(X_\sigma,f_\sigma)_{\sigma\in A}$, be a combinatorial Seshadri stratification on $X_P\subseteq \mathbb P(V)$
and denote by $\mathcal F_{\max}(A)$  the set of all maximal chains in $A$.
Let $\mathbb Q^A$ be the vector space with the standard basis $\{e_\tau\mid \tau\in A\}$.
Given a maximal chain $\mathfrak C\in \mathcal F_{max}(A)$, we identify the vector space
$\mathbb Q^\mathfrak C$ with the subspace of $\mathbb Q^A$ spanned by the basis elements $e_\tau$, $\tau\in \mathfrak C$.
So we view the valuation $\nu_\mathfrak C$ defined in Definition~\ref{def:valuation} 
as a map $\nu_\mathfrak C:\mathbb K(\hat X_P)\setminus \{0\}\rightarrow \mathbb Q^A$, 
such that the image lies in the subspace $\mathbb Q^\mathfrak C\subseteq  \mathbb Q^A$.
 
We fix on $A$  a linearization ``$>^t$'' of the partial order on $A$, i.e.``$>^t$'' is a total order on $A$ such that 
$\tau>\sigma$ for $\tau,\sigma\in A$ implies $\tau>^t \sigma$. We get on $\mathbb Q^A$ a total order by  
taking the induced lexicographic order, which makes $\mathbb Q^A$ into a totally ordered abelian group.

It is well known that the minimum function applied to a finite family of quasi-valuations is a quasi-valuation
(see, for example, \cite{FL}).
\begin{definition}\label{def:quasivaluation}
The \emph{quasi-valuation $\nu:\mathbb{K}[\hat{X_P}]\setminus\{0\}\to\mathbb{Q}^A$ }
associated to the combinatorial Seshadri stratification $(X_\sigma,f_\sigma)_{\sigma\in A}$ and the fixed 
total order $>^t$ on $A$ is the map  defined by:
$$
g\mapsto\nu(g):=\min\{\nu_{\mathfrak{C}}(g)\mid \mathfrak{C}\in \mathcal F_{\max}(A)\}.
$$
\end{definition}
Since $\nu_{\mathfrak{C}}(g^k)=k\nu_{\mathfrak{C}}(g)$ for all $g\in \mathbb{K}[\hat{X_P}]\setminus\{0\}$, $k\in\mathbb N$,
and all maximal chains in $A$, the quasi-valuation $\nu$ is \textit{homogenous}: $\nu(g^k)=k\nu(g)$ for $k\in\mathbb N$.
\begin{remark}
The quasi-valuation $\nu$ depends on the choice of the linearization ``$>^t$'', so it would
be more apt to write $\nu_{>^t}$ and to add ``$>^t$'' to the objects defined in the following. 
To avoid an excess of indexing we stick to the notation above. In addition,  
many of these objects turn out to be essentially independent of the choice of ``$>^t$''. 
\end{remark}
\subsection{First properties}
Let $\Gamma:=\{\nu(g)\mid g\in\mathbb{K}[\hat{X}_P]\setminus\{0\}\}\subseteq\mathbb{Q}^A$ be the image of the quasi-valuation. 
For $\underline{a}=\sum_{\tau\in A}a_\tau e_\tau\in \mathbb{Q}^A$, denote by $\mathrm{supp}\,\underline{a}:=\{\tau\in A\mid a_\tau\neq 0 \}$
the support of $\underline{a}$. By construction, the support of $\underline{a}\in\Gamma$ is always contained in some maximal chain $\mathfrak C$.
\begin{lemma}\label{quasi-T-invariant}
\begin{enumerate}
\item[{\it a)}] If $g=g_{\eta_1}+\ldots+g_{\eta_q}$ is a decomposition of $g\in\mathbb K[\hat X_P]$ into $\hat T$-eigenvectors,
then $\nu(g)=\min\{ \nu(g_{\eta_j})\mid   j=1,\ldots,q\}$.
\item[{\it b)}] $\Gamma:=\{\nu(g)\mid g\in\mathbb{K}[\hat{X}_P]\setminus\{0\}, g\textit{\ is a $\hat T$-eigenfunction}\}$.
\item[{\it c)}] $\nu(f_\tau)=e_\tau$ for all extremal functions $f_\tau$, $\tau\in A$, of the Seshadri stratification, independent
of the choice of the linearization ``$>^t$''.
\end{enumerate}
\end{lemma}
\begin{proof}
Let $g=g_{\eta_1}+\ldots+g_{\eta_q}$ be a decomposition of $g\in\mathbb K[\hat X_P]$ into $\hat T$-eigenvectors.
By definition we have 
$$
\nu(g)=\min\big\{ \min\{\nu_{\mathfrak C}(g_{\eta_i})\mid i=1,\ldots,t\}\vert    \mathfrak{C}\in \mathcal F_{\max}(A)\big\}=
\min\{ \nu(g_{\eta_j})\mid   j=1,\ldots,q\},
$$
which proves {\it a)} and {\it b)}. 
By Example~\ref{extremal:valuation} we know $\nu_{\mathfrak C}(f_\tau)=e_\tau$ if $\tau\in \mathfrak C$,
and Corollary~\ref{weight:and:degree:formula2} implies $\nu_{\mathfrak C}(f_\tau)>^t e_\tau$ if  $\tau\not\in \mathfrak C$,
independent of the choice of the linearization.
\end{proof}
\begin{lemma}\label{product:extremal:quasi11}
For a product $\prod_{\sigma\in A} f_\sigma^{m_\sigma}$ of extremal functions we have 
$\nu(\prod_{\sigma\in A} f_\sigma^{m_\sigma})
\ge^t \sum_{\sigma\in A} m_\sigma e_{\sigma}$, where equality holds if and only if there exists a maximal chain
 $\mathfrak C=(\sigma_r,\ldots,\sigma_0)$ such that $\{\sigma\in A\mid m_\sigma>0\}\subseteq \mathfrak C$.
 If $\nu(\prod_{\sigma\in A} f_\sigma^{m_\sigma})= \sum_{\sigma\in A} m_\sigma e_{\sigma}$, then the 
 equality holds independent of the choice of the linearization ``$>^t$''.
\end{lemma}
\begin{proof}
The quasi-additivity and the homogeneity of a quasi-valuation implies immediately
$\nu(\prod_{\sigma\in A} f_\sigma^{m_\sigma})\ge^t \sum_{\sigma\in A} m_\sigma \nu( f_\sigma)
=\sum_{\sigma\in A} m_\sigma e_{\sigma}$. So this is a lower bound for $\nu_{\mathfrak C}(\prod_{\sigma\in A} f_\sigma^{m_\sigma})$
for all maximal chains $\mathfrak C$, and this bound is independent of the choice of the linearization.

If there exists a maximal chain such that  $\{\sigma\in A\mid m_\sigma>0\}\subseteq \mathfrak C$, then the additivity
of a valuation implies: $ \nu_{\mathfrak C}(\prod_{\sigma\in A} f_\sigma^{m_\sigma})= \sum_{\sigma\in \mathfrak C} m_\sigma \nu_{\mathfrak C}( f_\sigma)
=\sum_{\sigma\in \mathfrak C} m_\sigma e_{\sigma}$, and hence we have equality also for the quasi-valuation,
independent of the choice of the linearization.

Now let $\mathfrak C'=(\tau_r,\ldots,\tau_0)$ be a maximal chain such that  
$\{\sigma\in A\mid m_\sigma>0\}\not\subseteq \mathfrak C'$.
We proceed by induction on $\sharp \{\sigma\in A\setminus \mathfrak C'\mid m_\sigma>0\}$.
If the number is equal to zero, there is nothing to prove:  $\nu_{\mathfrak C'}(\prod_{\sigma\in A} f_\sigma^{m_\sigma})
=\sum_{\sigma\in \mathfrak C} m_\sigma e_{\sigma}$.


Suppose now $\sharp \{\sigma\in A\setminus \mathfrak C'\mid m_\sigma>0\}=q\ge 1$ 
and let $\kappa\in A\setminus \mathfrak C'$ be such that $m_\kappa>0$.
We assume by induction $\nu_{\mathfrak C'}(\prod_{\sigma\in A\setminus\{\kappa\}} f_\sigma^{m_\sigma})
\ge^t \sum_{\sigma\in A\setminus\{\kappa\}} m_\sigma e_{\sigma}$.
The additivity of a valuation, induction and Corollary~\ref{weight:and:degree:formula2} implies:
$$
\begin{array}{rcl}
\nu_{\mathfrak C'}(\prod_{\sigma\in A} f_\sigma^{m_\sigma})&=&\nu_{\mathfrak C'}(\prod_{\sigma\in A\setminus\{\kappa\}} f_\sigma^{m_\sigma})
+\nu_{\mathfrak C'}(f_\kappa^{m_\kappa})\\
&\ge^t& \sum_{\sigma\in A\setminus\{\kappa\}} m_\sigma e_{\sigma}+\nu_{\mathfrak C'}(f_\kappa^{m_\kappa})\\
&>^t& \sum_{\sigma\in A} m_\sigma e_{\sigma}.
\end{array}
$$
It follows, independently of the choice of the linearization: if $\{\sigma\in A\mid m_\sigma>0\}\not\subseteq \mathfrak C'$,
then $\nu_{\mathfrak C'}(\prod_{\sigma\in A} f_\sigma^{m_\sigma})>^t  \sum_{\sigma\in A} m_\sigma e_{\sigma}$,
which finishes the proof of the lemma.
\end{proof}

\subsection{The quasi-valuation and  weight combinatorics}
Let $(X_\sigma,f_\sigma)_{\sigma\in A}$ be a combinatorial Seshadri stratification on $X_P\subseteq \mathbb P(V)$.
Denote by $\mathcal T=(\Delta_C)_{C\in\mathcal F(A)}$ the triangulation of $P$ indexed by flags
associated to the equivalence class of $(X_\sigma,f_\sigma)_{\sigma\in A}$ (Theorem~\ref{SeshStrat:triang:marking}).
Given a $\hat T$-eigenfunction $g\in\mathbb K[X]$ of weight $\lambda_g$, the triangulation suggests to us a preferred
class of maximal chains in $A$: the maximal chains $\mathfrak C$ such that $\frac{-\lambda_g}{\deg g}\in  \Delta_{\mathfrak C}$.
Indeed:
\begin{proposition}\label{prop:simplex}
Let $g\in\mathbb K[X_P]\setminus\{0\}$ be a $\hat T$-eigenfunction of weight $\lambda_g$ and let 
$\mathfrak C=(\sigma_r,\ldots,\sigma_0)$ be a maximal chain in $A$.
The following are equivalent:
\begin{itemize}
\item[i)] $\nu(g)=\nu_\mathfrak C(g)$;
\item[ii)]  $\frac{-\lambda_g}{\deg g}\in \Delta_{\mathfrak C}$;
\item[iii)] $\nu_\mathfrak C(g)\in \mathbb Q^\mathfrak C_{\ge 0}$.
\end{itemize}
In particular, $\nu(g)\in\mathbb Q^A_{\ge 0}$, and  $\nu(g)$ is independent of the choice of the linearization ``$>^t$''. 
\end{proposition}
\begin{proof}
If $\frac{-\lambda_g}{\deg g}\in \Delta_{\mathfrak C}$, then $\lambda_g$ is a non-negative linear
combination of the weights $\mu_{\sigma_i}$, $i=0,\ldots,r$, which implies $\nu_\mathfrak C(g)\in \mathbb Q^\mathfrak C_{\ge 0}$.
If $\nu_\mathfrak C(g)=(a_r,\ldots,a_0)\in \mathbb Q^\mathfrak C_{\ge 0}$, then
by Corollary~\ref{weight:and:degree:formula1}:
$$
\frac{-\lambda_g}{\deg g} = \frac{-1}{\sum_{j=0}^r a_j\deg f_{\sigma_j}}\left(\sum_{i=0}^r a_i\mu_{\sigma_i} \right)=
\sum_{i=0}^r \left(\frac{a_i\deg f_{\sigma_i}}{\sum_{j=0}^r a_j\deg f_{\sigma_j}}\right)\frac{- \mu_{\sigma_i}}{\deg f_{\sigma_i}}
\in \Delta_{\mathfrak C},
$$
which shows the equivalence of \textit{ii)} and \textit{iii)}. 

Given a $\hat T$-eigenfunction $g\in\mathbb K[X_P]\setminus\{0\}$ of weight $\lambda_g$, fix a maximal chain
$\mathfrak C=(\sigma_r,\ldots,\sigma_0)$ such that $\frac{-\lambda_g}{\deg g}\in \Delta_{\mathfrak C}$.
It follows that $\nu_{\mathfrak C}(g)=(a_r,\ldots,a_0)$ consists only of non-negative numbers. So if we fix $k$ as in 
Corollary~\ref{weight:and:degree:formula1}, then the equality $g^k=cf_{\sigma_r}^{ka_r}\cdots f_{\sigma_0}^{ka_0}$
for some $c\in\mathbb K^*$ holds in $\mathbb K[\hat X_P]$, and hence 
$$\nu(g)=\frac{1}{k}\nu(f_{\sigma_r}^{ka_r}\cdots f_{\sigma_0}^{ka_0}) =\sum_{j=0}^ra_j e_{\sigma_j}=\nu_{\mathfrak C}(g)$$
by Lemma~\ref{product:extremal:quasi11}. This proves {\it ii)} implies {\it i)}, and
 $\nu(g)\in\mathbb Q^A_{\ge 0}$  is independent of the choice of ``$>^t$''.
Suppose $\mathfrak C=(\tau_r,\ldots,\tau_0)$ is a maximal chain such that $\nu(g)=\nu_{\mathfrak C}(g)$.
By the above we know $\nu(g)\in \mathbb Q^\mathfrak C_{\ge 0}$, and hence 
$\nu_\mathfrak C(g)\in  \mathbb Q^\mathfrak C_{\ge 0}$  which shows {\it i)} implies {\it iii)}.
\end{proof}
Let $g\in\mathbb K[X_P]\setminus\{0\}$ be a $\hat T$-eigenfunction of weight $\lambda_g$ and let 
$\mathfrak C=(\sigma_r,\ldots,\sigma_0)$ be a maximal chain in $A$ such that
$\frac{-\lambda_g}{\deg g}\in \Delta_{\mathfrak C}$. The value of the quasi-valuation in $g$ is completely determined
by its $\hat T$-weight:
\begin{corollary}\label{coro:explicit:weight:valuation}
If $\frac{-\lambda_g}{\deg g}=\sum_{\sigma\in\mathfrak C} a_\sigma \frac{-\mu_\sigma}{\deg f_\sigma}$
is an expression of $\frac{-\lambda_g}{\deg g}$ as a convex linear combination of the vertices of $\Delta_{\mathfrak C}$,
then $\nu(g)=\sum_{\sigma\in\mathfrak C}  \frac{a_\sigma \deg\, g}{\deg f_\sigma} e_\sigma$
\end{corollary}
One can go the other way round too: the quasi-valuation determines the weight and the degree. Let 
$g\in\mathbb K[\hat{X}_P]\setminus\{0\}$ be a $\hat T$-eigenfunction of weight $\lambda_g$. Since $\nu(g)=\nu_{\mathfrak C}(g)$
for some maximal chain, we can apply Corollary~\ref{weight:and:degree:formula1} and get:

\begin{corollary}\label{coro:explicit:quasi:valuation:gives:weight}
If $\nu(g)=(a_\sigma)_{\sigma\in A}$, 
then $\deg g=\sum_{\sigma\in A} a_\sigma\deg f_{\sigma}$ and $\lambda_g=\sum_{\sigma\in A} 
a_\sigma\mu_\sigma$, and there exists a positive integer $k>0$ and $c\in\mathbb K^*$ such that 
$g^k=c\prod_{\sigma\in A} f^{ka_\sigma}_{\sigma}$.
\end{corollary}
\subsection{Two fans of monoids}\label{2:fans:monoids}
The explicit description of the quasi-valuation in Corollary~\ref{coro:explicit:weight:valuation} in terms of weight combinatorics
and triangulations makes it possible to give a similar description of the image of the quasi-valuation
$\Gamma:=\{\nu(g)\mid g\in\mathbb{K}[\hat{X_P}]\setminus\{0\}\}\subseteq\mathbb{Q}^A$.
A quasi-valuation is only quasi-additive: $\nu(gh)\ge^t \nu(g)+\nu(h)$, so, in general, $\Gamma$ is not a monoid.
\begin{lemma}\label{monoid:structure}
If $g,h\in \mathbb K[\hat X_P]\setminus\{0\}$ are $\hat T$-eigenfunctions of weight $\lambda_g$ respectively $\lambda_h$,
then $\nu(gh)=\nu(g)+\nu(h)$ if and only if there exists a maximal chain $\mathfrak C$ such that 
 $\frac{-\lambda_g}{\deg g}, \frac{-\lambda_h}{\deg h}\in\Delta_\mathfrak C$.
\end{lemma}
\begin{proof}
The quasi-valuation is homogeneous, so we can replace $g$ and $h$ by a $k$-th power,
where $k$ is chosen as in Corollary~\ref{coro:explicit:quasi:valuation:gives:weight}, i.e.
we can replace without loss of generality $g$ and $h$ by   products of extremal functions, say 
$g=c_1\prod_{\sigma\in \textrm{supp\,}\nu(g)} f_\sigma^{m_\sigma}$, 
$h=c_2\prod_{\tau\in \textrm{supp\,}\nu(h)} f_\tau^{n_\tau}$. Here $c_1,c_2\in \mathbb K^*$.

Now the condition $\frac{-\lambda_h}{\deg h}, \frac{-\lambda_h}{\deg h}\in\Delta_\mathfrak C$ 
is by Proposition~\ref{prop:simplex} equivalent to the existence of a maximal chain $\mathfrak C$ such that
$\textrm{supp\,}\nu(g),\textrm{supp\,}\nu(h)\subseteq \mathfrak C$,
so we can apply Lemma~\ref{product:extremal:quasi11} to the product of extremal functions, which finishes the proof.
\end{proof}
Let $g,h\in \mathbb{K}[\hat{X_P}]\setminus\{0\}$. After rewriting both as a sum of $\hat T$-eigenvectors,
one concludes by Lemma~\ref{monoid:structure}:
\begin{corollary}\label{coro:support:additive}
 $\nu(gh)=\nu(g)+\nu(h)$ if and only if there exists a maximal chain $\mathfrak C$ such that 
 $\textrm{supp\,}\nu(g),\textrm{supp\,}\nu(h)\subseteq \mathfrak C$.
\end{corollary}

Recall that $\mathcal T=(\Delta_C)_{C\in\mathcal F(A)}$ denotes the triangulation of $P$ indexed by flags
associated to the equivalence class of $(X_\sigma,f_\sigma)_{\sigma\in A}$.
For a chain $C\in \mathcal F(A)$ let $K(\Delta_C)$ be the cone over $\Delta_C$ and set $S_C=S\cap K(\Delta_C)$.
The union of the cones $K(\Delta_C)$, $C\in \mathcal F(A)$, (together with the origin $\{0\}$ as cone over $\Delta_C$ for
$C$ the empty chain) form a fan. In the same way the union of the $S_C$, $C\in \mathcal F(A)$, forms a fan of monoids,
i.e. for all $C,C',C''\in \mathcal F(A)$ one has: 
$S_C\subseteq S_{C'}$ if and only if $C\subseteq C'$, and $S_C\cap S_{C'}=S_{C''}$,
where $C''=C\cap C'$.
We write $S_\mathcal T$ for this fan of monoids. As a set one has $S_{\mathcal T}=S$, 
but as operation $\hat +$ we have:  $\lambda\hat+\eta = \lambda+\eta$ if there 
exists a chain such that $\lambda,\eta\in S_C$, and $\lambda\hat+\eta$ is not defined otherwise.

The quasi-valuation provides a similar construction in $\mathbb Q^A$. 
For a chain $C\in \mathcal F(A)$ we replace the cone $K(\Delta_C)\subseteq \hat M_\mathbb Q$ by
the cone $K_C\subseteq \mathbb R^A$ spanned by the basis vectors $\{e_\sigma\mid \sigma\in C\}$. 
The collection of cones $\{K_C\mid  C \in \mathcal F(A)\}$ defines a fan in $\mathbb R^A$. 
Its maximal cones are the cones $K_\mathfrak C$ associated
to the maximal chains $\mathfrak C$ in $A$.
For a chain $C\in \mathcal F(A)$ denote by $\Gamma_{C}$ the subset 
$\Gamma_{C}=\{\underline{a}\in\Gamma\mid \mathrm{supp}\,\underline{a}\subseteq C\,\}\subseteq K_C$.
By Lemma~\ref{monoid:structure}, $\Gamma_{C}$ has a natural structure as a monoid, which makes
$\Gamma=\bigcup_{ C\in \mathcal F(A)} \Gamma_{C}$ into a fan of monoids.

\begin{theorem}\label{theorem:finite:generation}
\begin{itemize}
\item[{\it i)}] For all $C\in \mathcal F(A)$, $\Gamma_{C}$ is a finitely generated monoid. 
\item[{\it ii)}] The image of the quasi-valuation
$\Gamma:=\{\nu(g)\mid g\in\mathbb{K}[\hat{X_P}]\setminus\{0\}\}$ is a fan of finitely generated monoids.
\item[{\it iii)}] $\Gamma$ is, as a fan of monoids, isomorphic to $S_{\mathcal T}$.
$\Gamma$ is independent of the choice of ``$>^t$'', and equivalent Seshadri stratifications yield isomorphic 
fans of monoides.
\end{itemize}
\end{theorem}
\begin{proof}
The natural map $\Gamma\rightarrow  S$, which sends a tuple
$ \underline{a}\in\Gamma$ to the weight $\sum_{\sigma\in A} a_\sigma  \mu_\sigma$,
is by Corollary~\ref{coro:explicit:weight:valuation} and Corollary~\ref{coro:explicit:quasi:valuation:gives:weight}
a bijection. By Lemma~\ref{monoid:structure}, it is a morphism between the fans of monoids $\Gamma$ and $S_{\mathcal T}$. 
Since $S_{\mathcal T}$ depends only on the triangulation $\mathcal T$, this proves \textit{iii)} by 
Theorem~\ref{SeshStrat:triang:marking}.

Since $K(\Delta_C)$ is a rational polyhedral cone,
the monoid $S_C=K(\Delta_C)\cap \hat M=K(\Delta_C)\cap S$ is finitely generated 
(by Gordan's Lemma). The isomorphism above sends $\Gamma_C$ onto  $S_C$, which implies
that $\Gamma_C$ is finitely generated, which proves \textit{i)} and \textit{ii)}.
\end{proof}
\subsection{The associated graded algebra and the fan algebra}\label{sec:fan:algebra}
\begin{definition}\label{Defn:FanAlgebra}
The \emph{fan algebra} $\mathbb K[\Gamma]$ associated to the fan of monoids $\Gamma$ is defined as
$\mathbb K[\Gamma]:=\mathbb K[y_{\underline{a}}\mid \underline{a}\in \Gamma] / I(\Gamma)$,
where $I(\Gamma)$ is the ideal generated by the following elements:
\begin{itemize}
\item[{\it a)}] $y_{\underline{a}}\cdot y_{\underline{b}}-y_{\underline{a}+\underline{b}}$ if there exists a chain $C\subset A$ such that 
$\underline{a},\underline{b}\in K_C\subseteq \mathbb Q^A$,
\item[{\it b)}] $y_{\underline{a}}\cdot y_{\underline{b}}$ if there exists no such a chain.
\end{itemize}
\end{definition}

To simplify the notation, we will write $y_{\underline{a}}$ also for its class in $\mathbb{K}[\Gamma]$.
For a  chain $C$ let $\mathbb K[\Gamma_{C}]$ be the subalgebra:
$\mathbb K[\Gamma_{C}]:=\bigoplus_{\underline{a}\in\Gamma_{C}} \mathbb Ky_{\underline{a}}\subseteq \mathbb K[\Gamma]$.
The algebra $\mathbb K[\Gamma_{C}]$ is naturally isomorphic to the usual semigroup algebra associated to the monoid $\Gamma_{C}$.

We endow the algebra $\mathbb{K}[\Gamma]$ with a $\mathbb N$-grading inspired by Corollary~\ref{coro:explicit:quasi:valuation:gives:weight}: 
for $\underline{a}\in\mathbb Q^A$, the degree of $y_{\underline{a}}$ is defined by:
$\deg y_{\underline{a}}=\sum_{\sigma\in A} a_\sigma\deg f_\sigma$.

The quasi-valuation defines a filtration on $\mathbb K[\hat X_P]$ given by ideals. We set for $\underline{a}\in \Gamma$:
$$
\mathbb K[\hat X_P]_{\ge^t \underline{a}}=\{ g\in\mathbb K[\hat X_P]\mid  \nu(g)\ge^t \underline{a}\}, \quad 
\mathbb K[\hat X_P]_{>^t \underline{a}}=\{ g\in\mathbb K[\hat X_P]\mid  \nu(g) >^t \underline{a}\}.
$$
Denote by $\mathrm{gr}_{\nu} \mathbb K[\hat X_P]=\bigoplus_{\underline{a}\in \Gamma} 
\mathbb K[\hat X_P]_{\ge^t \underline{a}}/\mathbb K[\hat X_P]_{>^t \underline{a}}$ the associated graded algebra.
\begin{theorem}\label{fanAndDegeneratetheorem} 
The associated graded algebra $\mathrm{gr}_{\nu} \mathbb K[\hat X_P]$ is isomorphic to the fan algebra $\mathbb K[\Gamma]$.
In particular, it is independent of the choice of the linearization and it depends, up to isomorphism, only on the triangulation 
$\mathcal T$
associated to the equivalence class of the Seshadri stratification $(X_\sigma, f_\sigma)_{\sigma\in A}$.

The variety $X_0=\mathrm{Proj\,} (\mathbb K[\Gamma])$ is 
reduced, it is the irredundant union of the toric varieties $X_{\mathfrak C}=\mathrm{Proj\,} (\mathbb K[\Gamma_\mathfrak C])$,
where $\mathfrak C$ is running over the set of all maximal chains in $A$. The variety is equidimensional,
all irreducible components of $X_0$ have same dimension as $X_P$.
\end{theorem}
\begin{proof}
The classes $\{\bar f_{m,\eta}\mid (m,\eta)\in S\}$ of the basis elements of $\mathbb K[\hat X_P]$ (Lemma~\ref{explicit:basis})
form a basis for the associated graded algebra $\mathrm{gr}_{\nu} \mathbb K[\hat X_P]$. 
We have a natural map $\pi$ between the basis of $\mathbb K[\Gamma]$ and the basis of 
$\mathrm{gr}_{\nu} \mathbb K[\hat X_P]$:
it sends $y_{\underline a}$, $\underline a\in\Gamma$ to $\bar f_{m,\eta}$, where $(m,\eta)= \sum_{\sigma\in A} a_\sigma \mu_\sigma$. 
This map extends linearly to a vector space isomorphism $\pi:\mathbb K[\Gamma]
\rightarrow \mathrm{gr}_{\nu} \mathbb K[\hat X_P]$, which by Lemma~\ref{monoid:structure} is an algebra isomorphism.

The algebra $\mathbb K[\Gamma]$ has no nilpotent elements, so $X_0=\textrm{Proj\,} (\mathbb K[\Gamma])$ is 
reduced. Set $y_{\mathfrak C}=\prod_{\sigma\in\mathfrak C} y_{e_\sigma}$ and let $I_\mathfrak C$ be the annihilator
of $y_{\mathfrak C}$ in $\mathbb K[\Gamma]$. The quotient $\mathbb K[\Gamma]/I_\mathfrak C$ is isomorphic to
$\mathbb K[\Gamma_\mathfrak C]$, an algebra which has no zero-divisors. Hence $I_\mathfrak C$ is a prime ideal. It also follows that the intersection
$\bigcap_{\mathfrak C\in\mathcal F_{\max}(A)}I_{\mathfrak C}=(0)$ is the zero ideal.

The ideal $I_\mathfrak C$ is a minimal prime ideal: suppose $I\subsetneq I_\mathfrak C$ is an ideal.
Then $\mathbb K[\Gamma]/I$ contains an element $\bar g\not=0$ such that $\nu(g)\not\in \Gamma_\mathfrak C$,
and hence $y_{\mathfrak C}\bar g=0$ by  Corollary~\ref{coro:support:additive}. So the quotient has a zero divisor
and $I$ is hence not a prime ideal. It follows that $\bigcap_{\mathfrak C}I_\mathfrak C=(0)$
is the minimal prime decomposition of the zero ideal in $\mathbb K[\Gamma]$. For a maximal chain $\mathfrak C'$, 
$y_{\mathfrak C'}$ is a non-zero element in the intersection 
$\bigcap_{\mathfrak C\not=\mathfrak C'}I_\mathfrak C$. This shows that the intersection $\bigcap_{\mathfrak C}I_\mathfrak C$ is non-redundant.

An irreducible component of $X_0$ is hence isomorphic to $X_{\mathfrak C}=\textrm{Proj\,} (\mathbb K[\Gamma_\mathfrak C])$
for some maximal chain $\mathfrak C=(\tau_r,\ldots,\tau_0)$. By definition, the functions $y_{e_{\tau_i}}$, $0\le i\le r$,
are algebraically independent and all other functions $y_{\underline a}$, $\underline a\in\Gamma_{\mathfrak C}$
depend algebraically on these functions. It follows $\dim X_{\mathfrak C}=
\dim \textrm{Proj\,} (\mathbb K[\Gamma_\mathfrak C])=r=\dim X_P$.
\end{proof}

\section{A flat degeneration induced by a $\mathbb G_m$-action}\label{sec:one:parameter}
Let $\texttt{G}_1=\{f_{1,\chi_1},\ldots,  f_{1,\chi_r}\}$ be the degree $1$ elements in the basis 
(see Lemma~\ref{explicit:basis}) of  $\mathbb K[\hat X_P]$. They generate $\mathbb K[\hat X_P]$,
but, in general, their classes $\bar f_{1,\chi_1},\ldots, \bar f_{1,\chi_r}$ do not  generate
$\mathrm{gr}_\nu\mathbb K[\hat X_P]$. So to describe a flat degeneration $X_P\rightsquigarrow X_0$, 
we replace the given embedding $\hat\iota:\hat X_P\hookrightarrow  V$
(see Section~\ref{sec:toric:variety}) by an embedding $\hat X_P\hookrightarrow V\oplus U$ into a larger space. 
\begin{example}
We take the same polytope $P\subseteq \mathbb R^2$ and lattice $M$ as in Example~\ref{square2x2},
with the same marking except for the edge $\sigma$ joining the vertices $(0,2)$ and $(2,2)$,
here we take as marking the point $(\frac{4}{3},2)$.
Let $\mathcal T_\mathfrak m$ be the associated triangulation, let $\mathfrak C$ be the maximal chain
starting with the vertex $(2,2)$, the edge joining  $(0,2)$ and $(2,2)$ and $P$ as maximal element.
Denote by $\Delta_\mathfrak C$ the corresponding simplex, the vertices are
$(2,2),(\frac{4}{3},2)$ and $(1,1)$. The points $(2,3,4)$ and $(3,4,6)$ are   
elements in $S_\mathfrak C=S\cap K(\Delta_\mathfrak C)$,
but they are not elements of the submonoid generated by 
$\{(1,a,b)\mid (a,b)\in\Delta_\mathfrak C\cap M\}=\{(1,1,1),(1,2,2)\}$.
So by the multiplication rules in $\mathrm{gr}_\nu\mathbb K[\hat X_P]$ (see Definition~\ref{Defn:FanAlgebra},
Theorem~\ref{theorem:finite:generation} and Theorem~\ref{fanAndDegeneratetheorem}),
to get a generating system for $\mathrm{gr}_\nu\mathbb K[\hat X_P]$, one has to add at least the classes 
$\bar f_{(2,3,4)}$ and $\bar f_{(3,4,6)}$.
\end{example}
We add to $\texttt{G}_1$ some higher degree elements $\texttt{G}=\texttt G_1\cup \{f_{m_{r+1},\chi_{r+1}},\ldots,  f_{m_p,\chi_p}\}$ 
taken from the basis (Lemma~\ref{explicit:basis}) so that $\overline{\texttt{G}}=\{\bar f_{m,\chi} \mid  f_{m,\chi}\in \texttt{G}\}$ is a generating system for $\mathrm{gr}_\nu\mathbb K[\hat X_P]$. 

Note that for all maximal chains $\mathfrak C$ holds: 
$\overline{\texttt{G}}_\mathfrak C=\{\bar f_{m,\chi} \mid  f_{m,\chi}\in \texttt{G}, \mathrm{supp}\nu(f_{m,\chi})\subseteq \mathfrak C\}$
generates the subalgebra $\mathbb K[\Gamma_\mathfrak C]\subseteq  
\mathbb K[\Gamma]\simeq \mathrm{gr}_\nu\mathbb K[\hat X_P]$. 
By construction, one has hence for the algebra $\mathbb K[V\oplus U]=\mathbb K[x_1,\ldots,x_p]$ two 
surjective algebra morphisms:
\begin{equation}\label{morphisms}
\begin{array}{rcl}
\bar\theta:\mathbb K[x_1,\ldots,x_p]&\rightarrow & \mathrm{gr}_\nu \mathbb K[\hat X_P];\\
\forall\, i=1,\ldots,p:\ x_i&\mapsto &\bar f_{m_i,\chi_i};\\
\end{array}\quad\textrm{and}\quad
\begin{array}{rcl}
\theta:\mathbb K[x_1,\ldots,x_p]&\rightarrow&  \mathbb K[\hat X_P];\\
\forall\, i=1,\ldots,p:\ x_i&\mapsto&  f_{m_i,\chi_i}.\\
\end{array}
\end{equation}
and corresponding embeddings $\bar\Theta: \hat X_0=\textrm{Spec\,} (\mathbb K[\Gamma])\rightarrow V\oplus U$ and 
$\Theta: \hat X_P\rightarrow V\oplus U$.
Since $ \hat X_P$ is already embedded in $V$, here is another description of the morphism  $\Theta$:
\begin{equation}\label{new:embedding}
\Theta:\hat X_P\rightarrow V\oplus U, \quad x\mapsto (x, f_{m_{r+1},\chi_{r+1}}(x),\ldots,f_{m_p,\chi_p}(x)).
\end{equation}

\subsection{Weighted projective varieties}\label{Sec:5.1}
We recall  some notation, for more details we refer to the notes \cite{Ho}. 
We endow the polynomial ring $\mathbb K[x_1,\ldots,x_p]$ with a 
$\mathbb N$-grading defined by $\deg_{\underline m} x_i=m_i$ (in particular $m_1=\ldots=m_r=1$).
We denote the ``$\mathrm{Proj}$'' of the  $\deg_{\underline m}$-graded ring by 
$\mathrm{Proj}_{\underline m}\mathbb K[x_1,\ldots,x_p]$ to avoid
confusion with  the standard projective space $\mathbb P(\mathbb K^p)=\mathrm{Proj}(\mathbb K[x_1,\ldots,x_p])$.

We denote the ``$\mathrm{Proj}$'' of the  $\deg_{\underline m}$-graded ring by $\mathbb P(m_1,\dots,m_p)$ and call it \emph{weighted projective space}.
Denote by $\texttt{Gr}_{\underline{m}}\simeq \mathbb K^*$ the grading group acting on $\mathbb K^p$ by 
$$\xi\cdot(a_1,\ldots,a_p)=
(\xi^{m_1}a_1,\ldots,\xi^{m_p}a_p)$$ 
for $\xi\in\texttt{Gr}_{\underline{m}}.$
 A more geometric description of $\mathbb P(m_1,\ldots,m_p)$ is given as a quotient by the action of the grading group: 
$\mathbb P(m_1,\ldots,m_p)=(\mathbb K^p\setminus\{0\}) / \texttt{Gr}_{\underline{m}}$.

If $f\in \mathbb K[x_1,\ldots,x_p]$ is  $\deg_{\underline m}$-homogeneous, i.e.
$$f(\xi^{m_1}a_1,\ldots,\xi^{m_p}a_p)=\xi^d f(a_1,\ldots,a_p),$$ 
where $d=\deg_{\underline m}f$, then 
 $V(f)=\{[v]\in\mathbb P(m_1,\ldots,m_p)\mid f(v)=0 \}$ is well defined.

For a $\deg_{\underline m}$-homogeneous ideal $I\subset \mathbb K[x_1,\ldots,x_p]$
let $\hat Y=V(I)\subseteq \mathbb K^p$ be the vanishing set of $I$. We denote by
$Y=V(I)\subseteq \mathbb P(m_1,\ldots,m_p)$ the vanishing set of all 
$\deg_{\underline m}$-homogeneous $f$ in $I$. We call $Y$  
\emph{the weighted algebraic set} associated to $I$, and $\hat Y$ is called
the \emph{affine quasi-cone over $Y$}, note that $Y=(\hat Y\setminus\{0\})/\texttt{Gr}_{\underline{m}}$.
If $Y$ is irreducible, then $Y$ is called a \emph{weighted projective variety}. 
If $I$ is a radical ideal, then the graded ring $ \mathbb K[x_1,\ldots,x_p]/I$ is called the 
homogeneous coordinate ring $\mathbb K[\hat Y]$ of $Y=V(I)$.

The morphisms $\bar\theta$ and $\theta$ defined in \eqref{morphisms} send monomials of 
$\deg_{\underline m}$-degree $d$ to monomials in the $\bar f_{m_i,\chi_i}$ (respectively 
$ f_{m_i,\chi_i}$) of the same degree with respect to the standard grading on 
$\mathrm{gr}_\nu \mathbb K[\hat X_P]$ (respectively on $\mathbb K[\hat X_P]$). Both rings, 
$\mathrm{gr}_\nu\mathbb K[\hat X_P]$ and $\mathbb K[\hat X_P]$,
are reduced, so $\ker\bar\theta$, as well as $\ker\theta$, are $\deg_{\underline m}$-homogeneous radical ideals 
and define weighted algebraic sets in $\mathbb P(m_1,\ldots,m_p)$, isomorphic to $X_0$ respectively $X_P$.
Since $\ker\theta$ is a prime ideal, note that $X_P=V(\ker\theta)\subseteq \mathbb P(m_1,\ldots,m_p)$
is a weighted projective variety.

For a polynomial $g(x_1,\ldots,x_p)=\sum_{\alpha} b_\alpha x^\alpha$  we use the multi-index notation,
i.e. $\alpha=(\alpha_1,\ldots,\alpha_p)\in\mathbb N^p$ and $x^\alpha=x_1^{\alpha_1}\cdots x_p^{\alpha_p}$,
and for the polynomial ring $\mathbb K[x_1,\ldots,x_p]$ we just write $\mathbb K[\underline x]$.
Note that for a monomial $x^\alpha$ we have: $\deg_{\underline m}x^\alpha=\deg \theta(x^\alpha)$.

We endow $\mathbb K^p$ with a $\hat T$-action. Let $\{e_1,\ldots,e_p\}$ be the standard basis of $\mathbb K^p$, 
we set for $\hat t=(c,t)\in\hat T$: $\hat t\cdot e_i=c^{m_i}\chi_i(t) e_i$, $i=1,\ldots,p$. We get an induced $T$-action on the 
corresponding weighted projective space $\mathbb P(m_1,\ldots,m_p)$.

The polynomial ring $\mathbb K[\underline x]$ gets endowed with an induced $\hat T$-action by algebra isomorphisms, 
here $x_i$ becomes a $\hat T$-eigenfunction of weight $(-m_i,-\chi_i)$,
$i=1,\ldots,p$. So the morphisms $\bar\theta$ and $\theta$ are surjective, $\hat T$-equivariant and $\deg_{\underline m}$-preserving morphisms.

\subsection{A global monomial preorder}\label{filtration:and:basis}
We define a total order on $\mathbb N\times \mathbb Q^A_{\ge 0}$ by:
$(m',\underline a')\preceq (m,\underline a) \textrm{\ if\ }m'<m \textrm{\ or\ }m'=m,\, \underline a'\ge^t \underline a$.
We now endow the polynomial ring $\mathbb K[\underline x]$  with  an $\mathbb N\times \mathbb Q_{\ge 0}^A$-grading.
\begin{definition}\label{degAgrading}
The $\mathbb N\times \mathbb Q^A_{\ge 0}$-grading is defined by $\deg_{A} x_i= (m_i,\nu( f_{m_i,\chi_i}))$, $i=1,\ldots,p$,
and  $\deg_{A} 1=0$.
We introduce on the set of all monomials in $\mathbb K[\underline x]$ a binary relation:  $x^\alpha\succeq_A x^\beta$ if $\deg_{A}x^\alpha\succeq \deg_{A}x^\beta$.
\end{definition}
Since ``$\succeq$'' defines a total order on $\mathbb N\times \mathbb Q^A_{\ge 0}$,  the induced
binary relation ``$\succeq_A$'' on the set of monomails is a \emph{weak order} (or \emph{total preorder}).
By definition, $\deg_{A}$ is additive, i.e. $\deg_{A}x^\alpha x^\beta=\deg_{A}x^\alpha + \deg_{A}x^\beta$.
The additivity of $\deg_A$  implies that ``$\preceq_A$'' is compatible and cancellative with
the multiplication, i.e. if $x^\alpha, x^\beta,  x^\gamma$ are monomials, then
$$
x^\alpha\succ_A x^\beta \Leftrightarrow x^\alpha x^\gamma\succ_A x^\beta x^\gamma.
$$
It follows that ``$\prec_A$'' is a \emph{monomial preorder} (see \cite{KTv}). 
The total degree part of the order ensures in addition that $1\prec_A x_i$  for all $i = 1,\ldots,p$, and if 
$x^\alpha,x^\beta$ are monomials such that $x^\alpha\not= 1$,
then  $x^\alpha x^\beta \succ_A x^\beta$. 
So we have a \emph{global monomial preorder}, see \cite{KTv}. 
\begin{definition}\label{defn:initial:term}
The \emph{initial term} $\textrm{in}_\nu g$ of a non-zero polynomial $g\in\mathbb K[\underline x]$ is the  sum
of the greatest terms of $g$ with respect to the global monomial preorder ``$\succ_A$''. If $I\subseteq \mathbb K[\underline x]$
is an ideal, then denote by $\textrm{in}_\nu I$ the ideal generated by the elements $\textrm{in}_\nu g$, $g\in I$.
\end{definition}
\begin{remark}
If $f$ is $\deg_{\underline m}$-homogeneous, then so is $\textrm{in}_\nu f$. In particular,
if the ideal $I$ is $\deg_{\underline m}$-homogeneous, then so is the ideal $\textrm{in}_\nu I$.
\end{remark}
\subsection{Minimal lifts}
Let $\mathcal I$ be the subset 
$$\{(m,\underline a)\in \mathbb N\times \mathbb Q^A_{\ge 0}
\mid \underline a\in \Gamma, m=\sum_{\sigma\in A} a_\sigma\deg f_\sigma\}\subseteq \mathbb N\times \mathbb Q^A_{\ge 0}.$$
\begin{definition}
The map $\textrm{val}_\nu: \mathbb K[\underline x]\setminus\{0\}\rightarrow \mathcal I$ is defined 
for a monomial $g=x^\alpha$ by  $\textrm{val}_\nu g:=(\deg_{\underline m} g,\nu( \theta(g)))\in\mathcal I$. 
For a polynomial $g=\sum_{\alpha} b_\alpha x^\alpha$  we define  
$\textrm{val}_\nu g$ to be the maximum of the values of the summands: 
$\mathrm{val}_\nu g={\max}_{\succeq}\{\textrm{val}_\nu(x^\alpha)\mid b_\alpha\not=0\}$. 
\end{definition}

\begin{definition}
A monomial $\prod_{i=1}^p x_i^{\ell_i}\in \mathbb K[\underline x]$ is called a \emph{minimal monomial} if 
there exists a maximal chain $\mathfrak C$ in $A$ such that 
$\{ \nu(\theta(x_i))\mid 1\le i\le p\textrm{\,and\,} \ell_i>0\}\subseteq\Gamma_{\mathfrak C}$.
We call such a maximal chain a \emph{support chain} for the minimal monomial.
\end{definition}
By the properties of a quasi-valuation we know $\nu(\theta(x^\alpha))\ge^t \sum_{i=1}^p \alpha_i \nu(\theta(x_i))$
and hence 
\begin{equation}\label{THE:INEQUALITY}
\textrm{val}_\nu x^\alpha=(\deg_{\underline m} x^\alpha,\nu(\theta(x^\alpha)) \preceq \left(\deg_{\underline m} x^\alpha,\sum_{i=1}^p \alpha_i \nu(\theta(x_i))\right) = \deg_A x^\alpha.
\end{equation}
Together with Corollary~\ref{coro:support:additive} one has:
\begin{lemma}\label{val:deg}
For a monomial $x^\alpha\in \mathbb K[\underline x]\setminus\{0\}$ holds: $\mathrm{val}_\nu x^\alpha \preceq \deg_A  x^\alpha$,
and we have equality: $\mathrm{val}_\nu x^\alpha = \deg_A  x^\alpha$ if and only if the monomial $x^\alpha$ is minimal.
\end{lemma}

Consider for $(m,\eta)\in S$ the function $f_{m,\eta}\in\mathbb K[\hat X_p]$. 
We know $\nu(f_{m,\eta})\in\Gamma_\mathfrak C$ for some maximal chain $\mathfrak{C}$.
Let $1\le i_1<\ldots<i_\ell\le p$ be such that $\overline{\texttt{G}}_\mathfrak C=\{\bar f_{m_{i_1},\chi_{i_1}},\ldots,
\bar f_{m_{i_\ell},\chi_{i_\ell}}\}$.  By the assumptions made at the beginning of this section, we can write
$\nu(f_{m,\eta})$ as a $\mathbb N$-linear combination: $\nu(f_{m,\eta})=\sum_{j=1}^{\ell} b_i\nu(f_{m_{i_j},\chi_{i_j}})$.
By taking the coefficients as exponents, we find a monomial 
$\mathbf f_{m,\eta}:=x_{i_1}^{b_{i_1}}\cdots x_{i_\ell}^{b_{i_\ell}}\in \mathbb K[\underline x]$
with the following property: it is a minimal monomial
such that $\bar\theta(\mathbf f_{m,\eta})=\bar f_{m,\eta}\in \mathrm{gr}_\nu \mathbb K[\hat X_P]$ and 
$\theta(\mathbf f_{m,\eta})=f_{m,\eta}\in \mathbb K[\hat X_P]$. 

\begin{definition} \label{fixed:minimal:lift}
We fix for all  $(m,\eta)\in S$ such a  lift for $f_{m,\eta}$: 
$\mathbf f_{m,\eta}=x_{i_1}^{b_{i_1}}\cdots x_{i_\ell}^{b_{i_\ell}}\in \mathbb K[\underline x]$,
called \emph{the fixed minimal lift} for $f_{m,\eta}\in \mathbb K[\hat X_P]$,
$({m,\eta})\in S$.
\end{definition}
\subsection{A basis compatible with $\ker\bar\theta$}\label{basis:ker:bar:theta}
Let $\overline{\mathbb B}_1$ be the set of all monomials in $\mathbb K[\underline x]$ which are not minimal, and let
$\overline{\mathbb B}_3$ be the set of all fixed minimal lifts: $\overline{\mathbb B}_3=\{\mathbf f_{m,\eta}\mid (m,\eta)\in S\}$. Finally, we set 
$$
\overline{\mathbb B}_2=\left\{x^\alpha-\mathbf f_{m,\eta}
\left\vert\  \substack{(m,\eta)\in S,\ \bar\theta(x^\alpha)=\bar f_{m,\eta},\ \mathbf f_{m,\eta}\not=x^\alpha\\
x^\alpha \textrm{\ minimal monomial}\ }
\right.\right\}
$$

\begin{lemma}\label{basisA}
The union $\overline{\mathbb B}=\overline{\mathbb B}_1\cup \overline{\mathbb B}_2\cup \overline{\mathbb B}_3$ is a basis 
for $\mathbb K[\underline x]$ which is homogeneous with  the  $\deg_{\underline m}$-grading and  the $\deg_{A}$-grading.
In addition, $\overline{\mathbb B}_1\cup \overline{\mathbb B}_2$ is a basis for $\ker\bar\theta$, and the image of 
 $\overline{\mathbb B}_3$ is a basis for $\mathrm{gr}_\nu \mathbb K[\hat X_P]\simeq \mathbb K[\underline x]/\ker\bar\theta$.
\end{lemma}
\begin{proof}
The union $\overline{\mathbb B}=\overline{\mathbb B}_1\cup \overline{\mathbb B}_2\cup \overline{\mathbb B}_3$ clearly is a basis for $\mathbb K[\underline x]$.
By Theorem~\ref{fanAndDegeneratetheorem} (for $\overline{\mathbb B}_1$) and construction (for $\overline{\mathbb B}_2$), 
$\overline{\mathbb B}_1\cup \overline{\mathbb B}_2\subseteq \ker\bar\theta$, whereas
$\overline{\mathbb B}_3$ is mapped by $\bar\theta$  bijectively onto a basis of $\mathrm{gr}_\nu \mathbb K[\hat X_P]$ (see \emph{ibidem}).
It follows that $\overline{\mathbb B}_1\cup \overline{\mathbb B}_2$ is a basis for $\ker\bar\theta$.

The elements in $\overline{\mathbb B}_2$ are $\textrm{val}_\nu$-homogeneous, i.e. for $\mathbf f_{m,\eta}-x^\alpha \in\overline{\mathbb B}_2$
one has $\textrm{val}_\nu(\mathbf f_{m,\eta})=\textrm{val}_\nu(x^\alpha)$ because $\theta(\mathbf f_{m,\eta})=
\theta(x^\alpha )$. Since both monomials are minimal by assumption, one has in addition
$\deg_A \mathbf f_{m,\eta}=\textrm{val}_\nu(\mathbf f_{m,\eta})=\textrm{val}_\nu(x^\alpha)=\deg_Ax^\alpha$, so the elements
are also $\deg_A$-homogeneous. The elements in $\overline{\mathbb B}_1$ and $\overline{\mathbb B}_3$
are just monomials, so  the basis is compatible with 
the  $\deg_{\underline m}$-grading and  the $\deg_{A}$-grading.
\end{proof}
\subsection{A basis compatible with $\ker\theta$}\label{theta_basis1}
%
To get a basis of $\mathbb K[\underline x]$ compatible with the  morphism
$\theta: \mathbb K[\underline x]\rightarrow \mathbb K[\hat X_P]$
we slightly change $\overline{\mathbb B}_1$ and set:
$$
 {\mathbb B}_1:=\left \{ 
x^\alpha-\mathbf f_{m,\eta}\left\vert\,  x^\alpha\in\overline{\mathbb B}_1, 
\theta(x^\alpha)= f_{m,\eta},\mathbf f_{m,\eta}\textrm{\ fixed minimal lift}
\right.\right\}, {\mathbb B}_2:= \overline{\mathbb B}_2,  {\mathbb B}_3:=\overline{\mathbb B}_3.
$$

\begin{lemma}\label{ker:theta:basis}
The union $ {\mathbb B}=  {\mathbb B}_1\cup  {\mathbb B}_2 \cup {\mathbb B}_3$ is a basis for $\mathbb K[\underline x]$
such that $  {\mathbb B}_1 \cup   {\mathbb B}_2$ is a basis for $\ker\theta$, which is compatible with  
the $\deg_{\underline m}$-grading, and the image of 
 ${\mathbb B}_3$ is a basis for $\mathbb K[\hat X_P]\simeq \mathbb K[\underline x]/\ker\theta$.
\end{lemma}

\begin{proof}
For $x^\alpha\in \overline{\mathbb B}_1$ one has  
$\textrm{val}_\nu x^\alpha \prec \deg_A  x^\alpha$ (Lemma~\ref{val:deg}).
So if $x^\alpha-\mathbf f_{m,\eta}\in {\mathbb B}_1$, then, by the minimality of $\mathbf f_{m,\eta}$,
we have 
\begin{equation}\label{initial:calculation}
\deg_A  \mathbf f_{m,\eta} =\textrm{val}_\nu  \mathbf f_{m,\eta}= \textrm{val}_\nu x^\alpha 
\prec \deg_A  x^\alpha.
\end{equation} 
The switch from $\overline{\mathbb B}_1$ to ${\mathbb B}_1$ can be viewed
as a triangular base change, and hence ${\mathbb B}$ is a basis.
By construction, $\theta(f)=0$ for all $f\in   {\mathbb B}_1\cup   {\mathbb B}_2$,  
whereas $  {\mathbb B}_3$ is mapped by $\theta$ onto a 
basis of $\mathbb K[\hat X_P]$ (Lemma~\ref{explicit:basis}).
%
As a consequence we have:
$  {\mathbb B}_1\cup  {\mathbb B}_2$ is a basis for $\ker\theta$.

Since all the basis elements are homogeneous with respect to the $\deg_{\underline m}$-grading, 
the basis is compatible with the $\deg_{\underline m}$-grading.
\end{proof}

\begin{lemma}\label{in:bijection}
The map $  b\mapsto \mathrm{in}_\nu   b$, which sends an element $  b\in {\mathbb B}$
to its initial term, induces a bijection ${\mathbb B}\rightarrow \overline{\mathbb B}$ such that 
$\mathrm{in}_\nu  g\in \overline{\mathbb B}_j$ for $g\in   {\mathbb B}_j$, $j=1,2,3$.
\end{lemma}
\begin{proof}
The elements in $  {\mathbb B}_2\cup   {\mathbb B}_3$ are $\deg_{A}$-homogenous
and hence one has for $  b\in   {\mathbb B}_i$, $i=2,3$:
$\textrm{in}_\nu   b=  b\in \overline{\mathbb B}_i$, so $\textrm{in}_\nu$ is the identity map on 
 ${\mathbb B}_2=\overline{\mathbb B}_2$ and $ {\mathbb B}_3=\overline{\mathbb B}_3$.

Given $  b\in  {\mathbb B}_1$, say  
$  b=x^\alpha-\mathbf f_{m,\eta}$, its inital term 
is by \eqref{initial:calculation}: $\textrm{in}_\nu   b=x^\alpha\in \overline{\mathbb B}_1$. 
By the construction of ${\mathbb B}_1$, the map $\nu:{\mathbb B}_1\rightarrow \overline{\mathbb B}_1$ is a bijection.
\end{proof}
\begin{lemma}\label{initial:tilde:theta:equals:bar:theta}
We have $\mathrm{in}_\nu(\ker\theta)=\ker\bar\theta$.
\end{lemma}
\begin{proof}
By Lemma~\ref{in:bijection}, one has $\ker\bar\theta\subseteq \text{in}_\nu(\ker\theta)$.
Let $g\in \ker\theta$, we write $g=\sum_{j\in \mathcal S} c_j   b_j$ 
as a linear combination  of elements $ b_j\in   {\mathbb B}_1\cup   {\mathbb B}_2$, where $\mathcal S$ is some finite indexing set. 
Since the $\textrm{in}_\nu b$, $b\in {\mathbb B}$, are lineary independent,
one has $\textrm{in}_\nu g= \textrm{in}_\nu(\sum_{j\in \mathcal S} c_j \textrm{in}_\nu  b_j)=
\sum_{j\in \mathcal S'} c_j \textrm{in}_\nu  b_j$, where $\mathcal S'\subseteq \mathcal S$
is the subset of indices such that $c_j\not=0$ and $\textrm{in}_\nu  b_j$ is 
of maximal $\deg_A$-degree. In particular, $\textrm{in}_\nu g \in \ker\bar\theta$,
and hence $\text{in}_\nu(\ker\theta)=\ker\bar\theta$.
\end{proof}
\subsection{An approximation by a weight function}\label{weight:function}
As in the case of monomial orders,  global monomial preorders can be 
approximated by integral weight orders, see, for example, \cite{KTv}, \cite{KR} and references therein.

For $\alpha,\beta\in\mathbb Z^p$ let $\alpha\cdot \beta=\sum_{i=1}^p \alpha_i\beta_i$.
For $\lambda\in  \mathbb Z^p$ let ``$\succ_\lambda$'' be the corresponding integral weight order on $\mathbb K[\underline x]$
defined by $x^\alpha\succeq_\lambda x^\beta$ if $\lambda\cdot \alpha\ge \lambda\cdot \beta$.
The initial term $\textrm{in}_\lambda(g)$ of a nonzero polynomial and the initial ideal $\textrm{in}_\lambda I$
are defined as in Definition~\ref{defn:initial:term}: $\textrm{in}_\lambda(g)$  is the  sum
of the greatest nonzero terms of $g$ with respect to the weight order ``$\succ_\lambda$'', 
and $\textrm{in}_\lambda I$ is the ideal generated by the elements $\textrm{in}_\nu g$, $g\in I$.
Note if the ideal $I$ is $\deg_{\underline m}$-homogeneous, then so is the ideal $\textrm{in}_\lambda I$.
The following theorem holds for monomial preorders, we formulate it
here just for the monomial preorder induced by the quasi-valuation $\nu$.

\begin{theorem}\cite[Theorem 3.2]{KTv}\label{aproximation:theorem}
There exists an integral vector $\lambda=(\lambda_1,\ldots,\lambda_p)\in\mathbb Z^p$
such that $\mathrm{in}_\nu(\ker\theta)=\mathrm{in}_\lambda(\ker\theta)$.
\end{theorem}
This integral vector $\lambda$ can be used to define a linear $\mathbb G_m$-action on $\mathbb K^p$:
$s\cdot (\sum_{i=1}^p c_ie_i) = \sum_{i=1}^p s^{\lambda_i}c_ie_i$ for $s\in \mathbb K^*$.
For the corresponding  $\mathbb G_m$-action on $\mathbb K[\underline x]$ by algebra homomorphisms 
we have for a monomial: $s\cdot x^\alpha=s^{-\lambda\cdot\alpha}x^{\alpha}$. 
The $\mathbb G_m$-action on $\mathbb K^p$ commutes with the grading action of $\texttt{Gr}_{\underline m}$  on $\mathbb K^p$; 
so we get an induced action on $\mathbb P(m_1,\ldots,m_r)$. 

For a non-zero polynomial $f=\sum_{\alpha} c_\alpha x^\alpha\in \mathbb K[\underline x]$ set 
$\deg_\lambda f=\max\{\lambda\cdot \alpha\mid c_\alpha\not=0\}$. 
Let $\mathcal S$ be the finite index set of $\alpha$ such that $c_\alpha\not=0$, and set
$\mathcal S'=\{\alpha\in \mathcal S\mid \lambda\cdot \alpha<\deg_\lambda f\}$.
We get
\begin{equation}\label{normalized:Gm:action}
s^{\deg_\lambda f}(s\cdot f)=\textrm{in}_\lambda f + \sum_{\alpha\in \mathcal S'} s^{\deg_\lambda f-\lambda\cdot \alpha}x^\alpha.
\end{equation}
By the definition of an initial ideal with respect to the weight order ``$\succ_\lambda$''  
it follows hence that such an ideal is generated by $\mathbb G_m$-eigenfunctions.
In particular, the ideal  $\ker\bar\theta=\text{in}_\nu(\ker\theta)=\textrm{in}_\lambda(\ker\theta)$ 
(Lemma~\ref{initial:tilde:theta:equals:bar:theta}, Theorem~\ref{aproximation:theorem})
is  generated by $\mathbb G_m$-eigenfunctions and hence:
\begin{lemma}
$X_0\subseteq \mathbb P(m_1,\ldots,m_r)$ is a $\mathbb G_m$-stable subvariety.
\end{lemma}
For a subvariety $Y\subseteq  \mathbb P(m_1,\ldots,m_r)$ denote by $I(Y)$ the 
$\deg_{\underline m}$-homogeneous vanishing ideal, and for a $\deg_{\underline m}$-homogeneous ideal $I\subseteq \mathbb K[\underline x]$
let $V(I)\subseteq \mathbb P(m_1,\ldots,m_r)$ be the vanishing set of the ideal.
For $s\in\mathbb G_m$ we have 
$I(s\cdot Y)=s\cdot I(Y)$ and $s\cdot V(I)=V(s\cdot I)$.

Let $I$ be a $\deg_{\underline m}$-homogeneous ideal.  The ideal $\lim_{s\rightarrow 0} s\cdot I$ is the 
ideal in $\mathbb K[\underline x]$ generated by the limit of the rescaled function $\lim_{s\rightarrow 0} s^{\deg_\lambda f}(s\cdot f)$, $f\in I$.
Equation \eqref{normalized:Gm:action} implies: $\lim_{s\rightarrow 0} s\cdot I=\textrm{in}_\lambda I$.
This ideal is again $\deg_{\underline m}$-homogeneous
and hence we have  by Lemma~\ref{initial:tilde:theta:equals:bar:theta} and Theorem~\ref{aproximation:theorem}:  
\begin{equation}
\lim_{s\rightarrow 0} s\cdot \ker\theta=\textrm{in}_\lambda \ker\theta=\textrm{in}_\nu \ker\theta = \ker\bar\theta.
\end{equation}
\begin{definition}\label{def:toric:degeneration}
For a weighted projective subvariety $Y\subseteq  \mathbb P(m_1,\ldots,m_r)$ denote by $I(Y)$ the 
$\deg_{\underline m}$-homogeneous vanishing ideal.
We say that the weighted algebraic set $Y_0\subseteq  \mathbb P(m_1,\ldots,m_r)$ is \emph{a toric degeneration of $Y$ inside
$\mathbb P(m_1,\ldots,m_r)$} and write
$\lim_{s\rightarrow 0} s\cdot Y=Y_0$, if $Y_0=V( \textrm{in}_\lambda I(Y))$.
\end{definition}

Summarizing we have for $X_P=V(\ker\theta)$ and 
$X_0=V(\ker\bar\theta)=V(\textrm{in}_\lambda(\ker\theta))$:
\begin{theorem}
The variety $X_0\subseteq  \mathbb P(m_1,\ldots,m_r)$ is a toric degeneration of $X_P$ inside
the weighted projectice space $\mathbb P(m_1,\ldots,m_r)$: $\lim_{s\rightarrow 0} s\cdot X_P=X_0$.
\end{theorem}
\subsection{Homogenization and a flat degeneration}\label{homogenizatíon}
A more formal way to look at the results in Section~\ref{weight:function} is to use the $\lambda$-homogenization of an ideal 
(see, for example, \cite[Section~15.8]{E} or \cite[Section~4.3]{KR}). We have the affine quasi-cone $\hat X_P\subseteq \mathbb K^p$
embedded in $\mathbb K^p$, together with a $\mathbb G_m$-action on $\mathbb K^p$. 
We add a variable $u$ and extend the 
action of $\mathbb G_m$ to $\mathbb K^p\oplus\mathbb K$ by 
$s\cdot (\sum_{i=1}^p c_ie_i, c) = (\sum_{i=1}^p s^{\lambda_i}c_ie_i,sc)$ for $s\in \mathbb K^*$. 

We extend the $\mathbb N$-grading to $\mathbb K[\mathbb K^p\oplus \mathbb K]=\mathbb K[\underline x,u]$
by setting $\deg_{\underline m} u=0$. And we extend the action of the grading
group $\mathtt{Gr}_{\underline m}\simeq \mathbb K^*$ to $\mathbb K^p\oplus \mathbb K$ by letting 
$\mathtt{Gr}_{\underline m}$ act 
trivially on $\mathbb K$.   
The action of $\mathbb G_m$ on $\mathbb K^p\oplus\mathbb K$ induces an action on 
$\mathbb K[\underline x,u]$ by algebra isomorphisms and, since the action preserves 
the $\deg_{\underline m}$-grading, we get an induced action on 
$\mathrm{Proj}_{\underline m}(\mathbb K[\underline x,u])$.

The inclusion $\mathbb K[u]\hookrightarrow \mathbb K[\underline x,u]$
induces a morphism $\pi: \mathrm{Proj}_{\underline m}(\mathbb K[\underline x,u]) \rightarrow \mathbb A^1$
which is $\mathbb G_m$-equivariant with respect to the  $\mathbb G_m$-action on 
$\mathrm{Proj}_{\underline m}(\mathbb K[\underline x,u])$ and the $\mathbb G_m$-action on 
$\mathbb K$ by multiplication.
\begin{definition}
For a polynomial $f=\sum_{\alpha} a_\alpha x^\alpha\in \mathbb K[\underline x]$  
set $\deg_\lambda f=\max\{\lambda\cdot \alpha\mid a_\alpha\not=0\}$.
We define a new function $\breve{f}\in \mathbb K[x_1,\ldots,x_p,u]$, called the \emph{$\lambda$-homogenization 
of $f$}:
\begin{equation}\label{procedure}
\breve{f}= u^{\deg_\lambda f}f(u^{-\lambda_1}x_1,\ldots,u^{-\lambda_p}x_p)\in \mathbb K[x_1,\ldots,x_p,u].
\end{equation}
 For a $\deg_{\underline m}$-homogeneous ideal 
$I\subseteq  \mathbb K[\underline x]$ denote by $\breve{I}\subseteq  \mathbb K[\underline x,u]$
the ideal generated by all the elements $\breve{f}$, $f\in I$. 
\end{definition}
For the $\mathbb G_m$-action we get: $s\cdot \breve{f}=s^{-\deg_\lambda f}f$, so
the function $\breve{f}$ is \emph{$\mathbb G_m$-homogeneous}.
Note that $\breve{f}=\textrm{in}_\lambda f +uh$, where $h\in \mathbb K[\underline x,u]$. 
Moreover, if $f$ is $\deg_{\underline m}$-homogeneous in  $\mathbb K[\underline x]$, 
then so is $\breve{f}$ in  $\mathbb K[\underline x,u]$.
We apply this homogenization procedure to the elements of the basis $\mathbb B$ of $\mathbb K[\underline x]$
to get $\breve{\mathbb B}=\{\breve{b}\mid b\in\mathbb B\}$.
It is easy to see:
\begin{lemma}
$\breve{\mathbb B}=\breve{\mathbb B}_1\cup \breve{\mathbb B}_2\cup \breve{\mathbb B}_3$ is
a basis of  $\mathbb K[\underline x,u]$ as a  $\mathbb K[u]$-module,
where $\breve{\mathbb B}_2={\mathbb B}_2$, $\breve{\mathbb B}_3={\mathbb B}_3$ and $\breve{\mathbb B}_1=\left \{ 
x^\alpha- u^{\ell_{m,\eta}}\mathbf f_{m,\eta}\left\vert\,  x^\alpha\in \overline{\mathbb B}_1, 
\theta(x^\alpha)= f_{m,\eta}\right.\right\}$, where $\ell_{m,\eta}=\lambda\cdot \alpha-\deg_\lambda \mathbf f_{m,\eta}>0$.
\end{lemma}
We apply the homogenization procedure to the ideal $J=\ker\theta$. Let $\mathfrak X_P\subseteq \mathbb K^p\oplus\mathbb K$
be the affine algebraic set obtained as the zero set $V(\breve{J})$ of the ideal $\breve{J}$.
Since $\breve{J}$ is generated by $\mathbb G_m$-eigenfunctions, $\mathfrak X_P$ is a 
$\mathbb G_m$-stable subset of $\mathbb K^p\oplus\mathbb K$, and we have a $\mathbb G_m$-equivariant
morphism $\tilde\pi: \mathfrak X_P\rightarrow \mathbb A^1$.
\begin{lemma}
$\mathfrak X_P\subseteq \mathbb K^p\oplus\mathbb K$ is an affine variety with coordinate ring 
$\mathbb K[\underline x,u]/\breve{J}$. The variety is stable under the action of the grading group $\texttt{Gr}_{\underline{m}}$.
The union $\breve{\mathbb B}_1\cup \breve{\mathbb B}_2$
is a basis for $\breve{J}$ as $\mathbb K[u]$-module, and $\mathbb K[\underline x,u]/\breve{J}$ is a free
$\mathbb K[u]$-module with basis the image of $\breve{\mathbb B}_3$.
\end{lemma}
\begin{proof}
The ideal $J=\ker\theta$ is a prime ideal, which implies by \cite[Proposition 4.3.10]{KR} that
 $\breve{J}$ is a prime ideal and hence $\mathfrak X_P\subseteq \mathbb K^p\oplus\mathbb K$ is an affine variety with coordinate ring 
$\mathbb K[\underline x,u]/\breve{J}$. The ideal  $J=\ker\theta$  is $\deg_{\underline m}$-homogeneous, hence
so is $\breve{J}$, and $\mathfrak X_P$ is thus stable under the action of the grading group $\texttt{Gr}_{\underline{m}}$.

The union $\breve{\mathbb B}_1\cup \breve{\mathbb B}_2$ is contained in $\breve{J}$ by construction, so 
the image of $\breve{\mathbb B}_3$ in $\mathbb K[\underline x,u]/\breve{J}$ is a generating system over the ring 
$\mathbb K[u]$. Since $J\subset \mathbb K[\underline x]$ is a proper prime ideal, one knows that $u$ is not a 
zero divisor in $\mathbb K[\underline x,u]/\breve{J}$ \cite[Proposition 4.3.5 e)]{KR}. So given a linear dependence 
relation between elements in the image of $\breve{\mathbb B}_3$ with coefficients in $\mathbb K[u]$, one may 
assume without loss of generality that at least one coefficient has a non-zero constant term. But this would give 
at $u=0$ a non-trivial linear dependence relation between the elements in $\overline{\mathbb B}_3$, which would 
be a contradiction. So the image of $\breve{\mathbb B}_3$ is a $\mathbb K[u]$-basis
for $\mathbb K[\underline x,u]/\breve{J}$.
\end{proof}

It follows that $\mathfrak X_P$ is an affine variety with coordinate ring $\mathbb K[\underline x,u]/\breve{J}$,
and  the $\mathbb G_m$-equivariant morphism $\tilde\pi: \mathfrak X_P\rightarrow \mathbb A^1$ is a flat morphism.
Since  $\breve{J}$ is $\deg_{\underline m}$-homogeneous,
we get for $\breve{\mathfrak X}_P=(\mathfrak X_P\setminus\{0\})/\texttt{Gr}_{\underline{m}}$
an induced morphism $\pi:\breve{\mathfrak X}_P\rightarrow \mathbb A^1$.

\begin{theorem}\label{proposition:degenerate}
\begin{itemize}
\item[{\it i)}] The morphism $\pi: \breve{\mathfrak X}_P\rightarrow \mathbb A^1$ is flat.
\item[{\it ii)}] The fibre over $0$ is isomorphic to $X_0$.
\item[{\it iii)}] $\pi$ is trivial over $\mathbb{A}^1\setminus\{0\}$ with fibre isomorphic to $X_P$.
\end{itemize}
\end{theorem}
\begin{proof}
Since $\mathbb K[\underline x,u]/\breve{J}$ is a free module over $\mathbb K[u]$, it is in particular
a flat module, which implies by $\deg_{\underline m}$-homogeneity the first claim.
Part {\it ii)} and {\it iii)} follows by \cite[Section 3]{KTv}, where it has been shown (here applied to the case
$J=\ker\theta$):
$\mathbb K[\underline x,u]/(\breve{J},u)\simeq  \mathbb K[\underline x]/\textrm{in}_\lambda J$, 
and $ (\mathbb K[\underline x,u]/\breve{J})[u^{-1}]\simeq  (\mathbb K[\underline x]/J)[u,u^{-1}]$,
which finishes the proof.
\end{proof}

\section{The integral case}\label{Zhu}
In this section we assume the combinatorial Seshadri stratification $(X_\sigma,f_\sigma)_{\sigma\in A}$ 
arises from a situation as in Example~\ref{integral:case}, i.e. for all $\sigma\in A$: the extremal function 
$f_\sigma$ is of degree one, and the $\hat T$-weight  $\mu_\sigma$ of $f_\sigma$ is a lattice point in 
the relative interior of the face $\sigma$. The associated triangulation $\mathcal T$ has hence lattice points as vertices.

 In particular, for every maximal chain $\mathfrak C\subseteq A$
we have in the triangulation $\mathcal T$ a simplex $\Delta_\mathfrak C$ with lattice points as vertices, and hence 
a toric variety $X_{\Delta_\mathfrak C}\subseteq \mathbb P(\mathbb K^{\Lambda_\mathfrak C})$,
where $\Lambda_\mathfrak C=\Lambda\cap \Delta_\mathfrak C$. Via the inclusion $\mathbb K^{\Lambda_\mathfrak C}
\hookrightarrow V=\mathbb K^{\Lambda}$ we view the (not necessarily normal) toric varieties
 $X_{\Delta_\mathfrak C}$ as being embedded in $\mathbb P(V)$.
The following example shows a normal polytope with an integral triangulation having a non-normal simplex.

\begin{example}
Let $P\subseteq\mathbb{R}^3$ be the polytope with vertices $v_0=(0,0,0)$, $v_1=(6,0,0)$, $v_2=(0,6,0)$ and $v_3=(0,0,6)$. By \cite[Theorem 2.2.11]{CLS}, the polytope $P$ is normal.

We consider the triangulation of $P$ whose vertices are: the point $(2,1,1)$ and the barycenters of the proper faces.

The simplex $Q$ having vertices $v_0 = (0,0,0)$, $(2,1,1)$, $(0,2,2)$ and $(0,0,3)$ is not normal. We checked this using {\tt Macaulay2} with the following code:

\begin{verbatim}
A = transpose matrix {{0,0,0}, {2,1,1}, {0,2,2}, {0,0,3}}
Q = convexHull A
isNormal Q
\end{verbatim}

However, not all the simplices of the triangulation are non-normal; for example, the one with vertices $(0,0,6)$, $(0,3,3)$, $(2,2,2)$ and $(2,1,1)$ is normal.

Here is a picture of the polytope $P$ with the non-normal simplex $Q$ in orange.

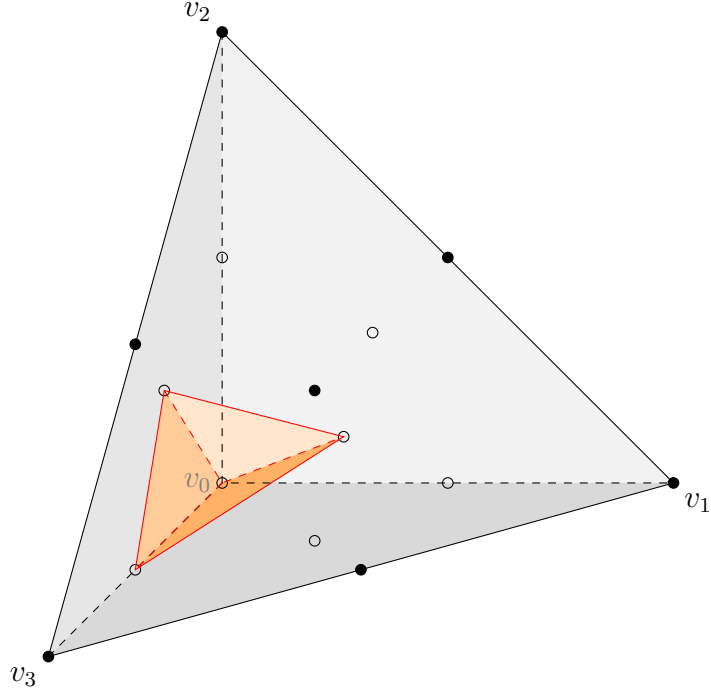
\[
\begin{tikzpicture}

	\fill[lightgray!20] (6,0,0) -- (0,0,0) -- (0,6,0);
	\fill[lightgray!40] (0,6,0) -- (0,0,0) -- (0,0,6);
	\fill[lightgray!60] (6,0,0) -- (0,0,0) -- (0,0,6);

	\fill[orange!20, opacity = 0.7] (2,1,1) -- (0,0,0) -- (0,2,2);
	\fill[orange!40, opacity = 0.7] (0,0,3) -- (0,0,0) -- (0,2,2);
	\fill[orange!60, opacity = 0.7] (2,1,1) -- (0,0,0) -- (0,0,3);

	\draw[draw=black] (6,0,0) -- (0,6,0) -- (0,0,6) -- cycle;
	\draw[black, dashed] (0,6,0) -- (0,0,0) -- (6,0,0);
	\draw[black, dashed] (0,0,0) -- (0,0,6);

	\node at (6,0,0) [below right] {$v_1$};
	\node at (0,6,0) [above left] {$v_2$};
	\node at (0,0,6) [below left] {$v_3$};
	\node at (0,0,0) [left] {\textcolor{gray}{$v_0$}};

	\draw[black] (3,0,0) circle (2pt); 
	\draw[black] (0,3,0) circle (2pt);
	\draw[black] (0,0,3) circle (2pt);

	\draw[black] (0,2,2) circle (2pt);
	\draw[black] (2,0,2) circle (2pt);
	\draw[black] (0,0,0) circle (2pt);
	\draw[black] (2,2,0) circle (2pt);

	\draw[black] (2,1,1) circle (2pt);

	\filldraw[black] (2,2,2) circle (2pt);
	\filldraw[black] (3,0,3) circle (2pt);
	\filldraw[black] (3,3,0) circle (2pt);
	\filldraw[black] (0,3,3) circle (2pt);
	\filldraw[black] (6,0,0) circle (2pt);
	\filldraw[black] (0,6,0) circle (2pt);
	\filldraw[black] (0,0,6) circle (2pt);

	\draw[red] (2,1,1) -- (0,0,3) -- (0,2,2) -- cycle;
	\draw[red, dashed] (0,2,2) -- (0,0,0) -- (2,1,1);
	\draw[red, dashed] (0,0,0) -- (0,0,3);


\end{tikzpicture}
\]
\end{example}
 
\subsection{A shadow}
The subspace  $V\subseteq V\oplus U$ is stable 
with respect to the  $\hat T$-, the $\mathbb G_m$- and the $\mathtt{Gr}_{\underline m}$-action 
on $V\oplus U$ (see Section \ref{Sec:5.1}), and the projection $\hat \phi:V\oplus U\rightarrow V$ is equivariant with respect to these actions. 
Recall that $\mathtt{Gr}_{\underline m}$ 
acts on $V$ by scalar multiplication, so $\hat O=\{(v,u)\mid v\in V,\ u\in U,\ v\not=0\}$ is an open and dense subset of $V\oplus U$,
stable with respect to the actions by  $\hat T$-, $\mathbb G_m$- and $\mathtt{Gr}_{\underline m}$. 
We get hence a $\hat T$- and $\mathbb G_m$-equivariant  rational map 
$\phi:\mathbb P(m_1,\ldots,m_p)\dashrightarrow \mathbb P(V)$, which is well defined 
on the open and dense subset $O=\{[v,u]\in \mathbb P(m_1,\ldots,m_p)\mid v\not=0\}$.

We write  $X_P^V\subseteq \mathbb P(V)$ to emphasize the fixed embedding of $X_P$ and to 
not confuse the embedding with  $X_P\hookrightarrow  \mathbb P(m_1,\ldots,m_p)$. Via the rational morphism
$\phi$ we can see the family of varieties $\{s\cdot X^V_P\mid s\in \mathbb G_m\}\subseteq \mathbb P(V)$
as a shadow of the family $\{s\cdot X_P\mid s\in \mathbb G_m\}\subseteq O$.
From this point of view we may think of 
$$
\mathbb X_0:=\lim_{s\rightarrow 0} s\cdot X^V_P\subseteq \mathbb P(V)
$$
as a shadow of the limit $X_0=\lim_{s\rightarrow 0} s\cdot X_P$ in $\mathbb P(m_1,\ldots,m_p)$.

The embedding $ X^V_P\subseteq \mathbb P(V)$ (Definition~\ref{def:toric:variety}) induces
a surjective algebra homomorphism $\theta_V: \mathbb K[V]\rightarrow \mathbb K[\hat X_P]$.
We view $V=\mathbb K^\Lambda\simeq\mathbb K^r$ as a subspace of $V\oplus U$ and write  $\mathbb K[x_1,\ldots,x_r]$ for $\mathbb K[V]$. The ideal $\ker\theta_V$ is the  vanishing ideal for $\hat X^V_P\subseteq \mathbb P(V)$, it is equal to
the intersection $\ker \theta\cap \mathbb K[x_1,\ldots,x_r]$.

The  integral vector $\lambda=(\lambda_1,\ldots,\lambda_p)\in\mathbb Z^p$ (Theorem~\ref{aproximation:theorem}) 
induces also a weight order
on $\mathbb K[x_1,\ldots,x_r]$. Let $\lambda^\dagger=(\lambda_1,\ldots,\lambda_r)\in\mathbb Z^r$ be the truncated 
vector, we define for $\alpha,\beta\in \mathbb Z^r$: $x^\alpha\succeq_{\lambda^\dagger} x^\beta$
if and only if $\lambda^\dagger\cdot \alpha\ge \lambda^\dagger\cdot \beta$. Note that 
$\lambda^\dagger\cdot \alpha=\lambda\cdot \alpha$ for all 
$x^\alpha\in \mathbb K[x_1,\ldots,x_r]\subseteq \mathbb K[x_1,\ldots,x_p]$.
Being translated into ideals, the task is to study $\lim_{s\rightarrow 0} s\cdot \ker\theta_V
= \textrm{in}_{\lambda^\dagger}\ker\theta_V$. Theorem~\ref{shadow:theorem} below
recovers for this special situation a result by Zhu \cite{Zhu}:

\begin{theorem}\label{shadow:theorem}
\begin{enumerate}
\item[{\rm i)}] We have $\sqrt{\mathrm{in}_{\lambda^\dagger} \ker\theta_V}=\ker\bar\theta\cap \mathbb K[x_1,\ldots,x_r]$.
\item[{\rm ii)}]  If we consider the toric degeneration of $X_P$ induced by the $\mathbb G_m$-action on $\mathbb P(V)$, then 
we get $\mathbb X_0=\lim_{s\rightarrow 0} s\cdot X_P=\bigcup X_{\Delta_\mathfrak C}$, where the union runs over
all maximal chains $\mathfrak C$ in $A$.
\end{enumerate}
\end{theorem}

Before we come to the proof, we want to point out that the integrality condition on the vertices of the 
triangulation $\mathcal T$ ensures that $\mathbb X_0$ is not too different from $X_0$. 
For a flag $C\subseteq A$ let  $K(\Delta_C)$ be the cone over the simplex 
$\Delta_C$  and let $S_C=K(\Delta_C)\cap S$ be the associated  monoid in the 
fan of monoids $S_{\mathcal T}$.
\begin{definition}\label{s:one:c}
We denote by $S_\mathfrak C^1\subseteq S_\mathfrak C$ the \emph{submonoid $S_C^1=\langle (1,\eta) \mid \eta\in \Delta_C\cap M\rangle_\mathbb N$ generated by the degree one elements}. 
Let $S^1$ be the fan of monoids obtained as the union $\bigcup_C S_C^1$, where $C$ 
is running over all flags $C\subseteq A$.
\end{definition}
Let $\mathfrak C$ be a maximal chain in $A$. Note that $S_\mathfrak C^1$ is the weight monoid of the embedded toric variety  
$X_{\Delta_\mathfrak C}\subseteq \mathbb P(V)$. So one can attach to a maximal chain $\mathfrak C$ two
affine toric varieties: $\hat X_{\Delta_\mathfrak C}=\textrm{Spec\,}\mathbb K[S_\mathfrak C^1]\subseteq V$, which is the
affine cone over $X_{\Delta_\mathfrak C}$,  and the weighted affine cone 
$\hat X_{\mathfrak C}=\textrm{Spec\,}\mathbb K[S_\mathfrak C]\subseteq V\oplus U$ 
over the irreducible component  $X_\mathfrak C\subseteq X_0$.
\begin{proposition}\label{finite:S1:monoid}
For all maximal chains $\mathfrak C$ in $A$, the 
morphism $\hat X_{\mathfrak C}\rightarrow \hat X_{\Delta_\mathfrak C}$, induced
by the inclusion of monoids $S_\mathfrak C^1\subseteq S_\mathfrak C$, is the normalization morphism.
\end{proposition}
\begin{proof}
The irreducible components $X_\mathfrak C\subseteq X_0$ are normal because $S_\mathfrak C$
is saturated. Let $\mathfrak A$ be the set $\{f_{m_i,\eta_i}\in G\mid r+1\le i\le p,\ \nu(f_{m_i,\eta_i})\in \mathfrak C\}$ and denote by
$\mathfrak a$ the set $\{\nu(f_{m,\eta})\mid f_{m,\eta}\in\mathfrak A\}$. 

Fix $k>0$ such that 
 $k\nu(f_{m,\eta})\in \mathbb N^A$ for all $f_{m,\eta}\in\mathfrak A$.
By Corollary~\ref{coro:explicit:quasi:valuation:gives:weight} we know for $f_{m,\eta}\in\mathfrak A$:
 $f_{m,\eta}^k$ is a product of the extremal weight vectors $f_\sigma$, $\sigma\in \mathfrak C$. 
 Since the $(1,\mu_\sigma)$, $\sigma\in A$, are elements in  $S_\mathfrak C^1$, 
 it follows that $k\nu(f_{m,\eta})\in S^1$, and hence:
 every element in $S_\mathfrak C$ can be written as a linear combination 
 of elements in $S_\mathfrak C^1$ and elements in $\mathfrak a$, 
 with non-negative integer coefficients, but where the
 coefficients of the elements in $\mathfrak a$ are bounded by $k$.
 It follows that $\mathbb K[S_\mathfrak C]$ is a finite $\mathbb K[S^1_\mathfrak C]$-module,
 and hence $\mathbb K[S_\mathfrak C]$ is integral over $\mathbb K[S^1_\mathfrak C]$,
 which finishes the proof.
\end{proof}
\subsection{Proof of Theorem~\ref{shadow:theorem}}

\begin{proof}
If $(m,\eta)\in S_\mathfrak C^1$ for some maximal chain $\mathfrak C$, then, by the definition
of $S_\mathfrak C^1$ (see Definition~\ref{s:one:c}) and Proposition~\ref{prop:simplex}, one can 
find a minimal lift (in the sense of Definition~\ref{fixed:minimal:lift})
which is an element in $\mathbb K[x_1,\ldots,x_r]$. In the following we assume without loss of generality
that we have fixed such a minimal lift $\mathbf f_{m,\eta}\in \mathbb K[x_1,\ldots,x_r]$ for all $(m,\eta)\in S_\mathfrak C^1$. 

Since $\textrm{in}_{\lambda^\dagger} f=\textrm{in}_{\lambda} f$ for 
$f\in\mathbb K[x_1,\ldots,x_r]\subseteq \mathbb K[x_1,\ldots,x_p]$, it follows for the initial ideals: 
$\textrm{in}_{\lambda^\dagger} \ker\theta_V\subseteq \textrm{in}_{\lambda}\ker \theta\cap \mathbb K[x_1,\ldots,x_r]$
and hence  $\textrm{in}_{\lambda^\dagger} \ker\theta_V\subseteq\ker\bar\theta\cap \mathbb K[x_1,\ldots,x_r]$. Moreover, 
since $\ker\bar\theta$ is a radical ideal:
$\sqrt{\textrm{in}_{\lambda^\dagger} \ker\theta_V}\subseteq\ker\bar\theta\cap \mathbb K[x_1,\ldots,x_r]$.

Let $x^\alpha\in \mathbb K[x_1,\ldots,x_r]$ be a monomial which is not minimal, so 
$x^\alpha\in\overline{\mathbb B}_1\cap \mathbb K[x_1,\ldots,x_r]$. Let 
$x^\alpha-\mathbf f_{m,\eta}\in \mathbb K[x_1,\ldots,x_p]$ be the corresponding
element in $\mathbb B_1$. If $(m,\eta)\in S^1$, then $\mathbf f_{m,\eta}\in \mathbb K[x_1,\ldots,x_r]$ and
hence $x^\alpha-\mathbf f_{m,\eta}\in \mathbb K[x_1,\ldots,x_r]$. It follows
$$\mathrm{in}_{\lambda^\dagger} (x^\alpha-\mathbf f_{m,\eta})=\mathrm{in}_{\lambda} (x^\alpha-\mathbf f_{m,\eta})=x^\alpha\in
{\mathrm{in}_{\lambda^\dagger} \ker\theta_V}.$$ 

If $(m,\eta)\not \in S^1$, then let $k>0$ be an integer
such that $k\nu(x^\alpha)\in\mathbb N^A$. 
It follows that $x^{k\alpha}$ is still an element of $\overline{\mathbb B}_1$,
but by Corollary~\ref{coro:explicit:quasi:valuation:gives:weight} (and Corollary~\ref{weight:and:degree:formula1}),
$\theta(x^{k\alpha})$ is equal to a product of extremal functions which 
are of degree one, and all of them have support in the same maximal chain. 
So if  $x^{k\alpha}-\mathbf f_{km,k\eta}$ is the corresponding element in $\mathbb B_1$,
then $(km,k\eta)\in S^1$, and hence $x^{k\alpha}\in \mathrm{in}_{\lambda^\dagger}\ker\theta_V$, which implies 
$x^{\alpha}\in \sqrt{\mathrm{in}_{\lambda^\dagger} \ker\theta_V}$. In other words:
$\overline{\mathbb B}_1\cap \mathbb K[x_1,\ldots,x_r]\subseteq \sqrt{\mathrm{in}_{\lambda^\dagger} \ker\theta_V}$.

Denote by 
$$\overline{\mathbb B}_2^1={\mathbb B}_2^1=
\{x^\alpha-\mathbf f_{m,\eta}\in {\mathbb B}_2 \mid x^\alpha\in\mathbb K[x_1,\ldots,x_n], (m,\eta)\in S^1\}.$$
By assumption, $\mathbf f_{m,\eta}\in\mathbb K[x_1,\ldots,x_n]$, and hence 
$$\mathrm{in}_{\lambda^\dagger} (x^\alpha-\mathbf f_{m,\eta})=\mathrm{in}_{\lambda} (x^\alpha-\mathbf f_{m,\eta})=
x^\alpha-\mathbf f_{m,\eta} \in \ker\theta_V.$$ It follows: $\overline{\mathbb B}_2^1\subseteq  
\sqrt{\textrm{in}_{\lambda^\dagger} \ker\theta_V}$.

To prove part {\it i)\,} of the Theorem, let $f\in \ker\bar\theta\cap \mathbb K[x_1,\ldots,x_r]$.
Since $\bar\theta$ is $\hat T$-equivariant, 
one can assume without loss of generality that $f$ is a 
$\hat T$-eigenfunction. So there exist an element $(m,\eta)\in S$ such that 
$f=\sum_{\alpha\in\mathcal S} c_\alpha x^\alpha$ is a finite linear combination
of monomials $x^\alpha\in\mathbb K[x_1,\ldots,x_r]$ such that $\bar\theta(x^{\alpha})=\bar f_{m,\eta}$ for
all $\alpha$. Here we assume that $\mathcal S$ is a finite index system and $c_\alpha\not=0$ for all 
$\alpha\in\mathcal S$. 

Since all non-minimal monomials are in  $\overline{\mathbb B}_1 \cap \mathbb K[x_1,\ldots,x_r]$ and hence
in $\sqrt{\textrm{in}_{\lambda^\dagger} \ker\theta_V}$ as well as in $\ker\bar\theta$,
one can assume in addition that $c_\alpha\not=0$ implies $x^{\alpha}$ is a minimal monomial.
So either $f=0$, which finishes the proof, or necessarily $(m,\eta)\in S^1$. Rewrite $f$ as
$$
f=c_\alpha \mathbf f_{m,\eta}+  \sum_{\beta\in\mathcal S\setminus\{\alpha\}} c_\beta x^\beta
\textrm{\ and set  }
\tilde f= \sum_{\beta\in\mathcal S\setminus\{\alpha\}} c_\beta (x^\beta-\mathbf f_{m,\eta}).
$$
By construction, $\tilde f$ is
a linear combination of elements in $ \mathbb B^1_{2,V}$, hence 
$\tilde f\in \sqrt{\textrm{in}_{\lambda^\dagger} \ker\theta_V}$
and  $f,\tilde f\in\ker\bar\theta\cap \mathbb K[x_1,\ldots,x_r]$.
This implies $f = \tilde f$ because otherwise $\mathbf f_{m,\eta}\in \ker\bar\theta$, 
which is not possible, hence $f\in \sqrt{\textrm{in}_{\lambda^\dagger} \ker\theta_V}$,
which finishes the proof of part {\it i)}.

To prove {\it ii)}, note that we have just shown:
$(\overline{\mathbb B}_1\cap \mathbb K[x_1,\ldots,x_r])\cup  \mathbb B^1_{2,V}$
is a vector space basis for $\sqrt{\textrm{in}_{\lambda^\dagger} \ker\theta_V}$ of $\hat T$-eigenfunctions. 
If we add to this set $ \mathbb B^1_{3,V}=\{\mathbf f_{m,\eta}\mid  (m,\eta)\in S^1\}$,
then we have a basis for $ \mathbb K[x_1,\ldots,x_r]$: a monomial in this ring is either 
not minimimal, and hence an element of $\overline{\mathbb B}_1\cap \mathbb K[x_1,\ldots,x_r]$;
or it is minimial, and then it is an element in the linear span of $\mathbb B^1_{2,V}\cup \mathbb B^1_{3,V}$.
It follows that the zero set $V(\sqrt{\textrm{in}_{\lambda^\dagger} \ker\theta_V})$ is the union of toric varieties,
where the irreducible components are indexed by maximal chains $\mathfrak C\subseteq A$ and 
the associated weight monoid is $S^1_\mathfrak C$, which finishes the proof: 
$\mathbb X_0=\lim_{s\rightarrow 0} s\cdot X_P=\bigcup X_{\Delta_\mathfrak C}$.
\end{proof}

\section{A geometric interpretation}\label{Sec:SS}
In this section we compare the construction of a combinatorial Seshadri stratification in this article 
with the construction of a Seshadri stratification in \cite{CFL}. We first recall the 
definition of a Seshadri stratification on an embedded projective variety $X\subseteq\mathbb{P}(V)$.
\subsection{Seshadri stratifications}
Let $V$ be a finite dimensional vector space over $\mathbb{K}$. The vanishing set of a homogeneous function $f\in\mathbb{K}(V^*)$ 
will be denoted by $\mathcal{H}_f:=\{[v]\in \mathbb{P}(V)\mid f(v)=0\}$. For an embedded projective subvariety $X\subseteq\mathbb{P}(V)$, we 
let $\hat{X}$ denote its affine cone in $V$.

Let  $X_p$, $p\in A$, be a finite collection of projective subvarieties of $X$ and 
$f_p\in\mathbb K[X]$, $p\in A$, be homogeneous functions of positive degrees. The index set $A$ inherits a poset structure by requiring: 
for $p,q\in A$, $p\geq q$ if and only if $X_p\supseteq X_q$. We assume that there exists a unique maximal element 
$p_{\max} \in A$ with $X_{p_{\max}} = X$.
We say that $\tau>\sigma$ is a covering relation if $\tau\ge \tau'>\sigma$ implies $\tau=\tau'$.
\begin{definition}
[\cite{CFL}]\label{Defn:SS1}
The collection of subvarieties $X_p$ and homogeneous functions $f_p$ for $p\in A$ is called a \emph{Seshadri stratification} on $X$, if the following conditions are fulfilled:
\begin{enumerate}
\item[(S1)] the projective subvarieties $X_p$, $p\in A$, are smooth in codimension one; if $q<p$ is a covering relation in $A$, then $X_q$ is a codimension one subvariety in $X_p$;
\item[(S2)] for $p,q\in A$ with $q\not\leq p$, the function $f_q$ vanishes on $X_p$;
\item[(S3)] for $p\in A$, it holds set-theoretically 
$$\mathcal{H}_{f_p}\cap X_p=\bigcup_{q\text{ covered by }p} X_q.$$
\end{enumerate}
The functions ${f_p}$ will be called \emph{extremal functions}.
\end{definition}


\begin{remark}\label{Compatible:subvariety}
A Seshadri stratifications of an embedded variety $X\subseteq \mathbb P(V)$ is compatible with
the subvarieties realizing the stratification:
for $p\in A$,  the poset $A_p=\{q\in A\mid q\le p\}$ has a unique maximal element. 
The collection of varieties $X_q\subseteq X_p$, $q\in A_p$, and the extremal functions 
$f_q\vert_{X_p}$, $q\in A_p$, satisfies the conditions (S1)-(S3), and hence defines a Seshadri stratification for $X_p\hookrightarrow  \mathbb P(V)$.
\end{remark}
%

\subsection{Seshadri stratifications on toric varieties which are $T$-equivariant}\label{Sec:toric0}
In the case of an embedded toric variety $X_P\subseteq \mathbb P(V)$ as in Section~\ref{sec:toric:variety}, 
it makes sense to consider only $T$-stable subvarieties and homogeneous $T$-eigenfunctions. 
Denote by $A$ the set of faces of the polytope $P$. Recall that this set is partially ordered.
\begin{definition}\label{defn:T:equivariant1}
A Seshadri stratification on $X_P$ is called \textit{$T$-equivariant} if
\begin{enumerate}
\item[(E1)] the collection of projective subvarieties of the Seshadri stratification consists of the $T$-orbit closures $X_\sigma$, $\sigma\in A$,
\item[(E2)] the collection of homogeneous functions of the Seshadri stratification $f_\sigma$, $\sigma\in A$, consists of homogeneous $T$-eigenfunctions.
\end{enumerate}
\end{definition}
In the case of toric varieties, the usual expectation is that all ``$T$-equivariant'' conditions on the variety and properties of the variety
can be  rephrased in terms of weight combinatorics. This holds also in the case of $T$-equivariant Seshadri stratifications, we recover
here the condition on the weights of the extremal functions in Definition~\ref{defn:T_equi:SS}:
\begin{theorem}\label{Seshadri:is:Seshadri:toric:theorem1}
Let $X_P\subseteq \mathbb P(V)$ be an embedded toric variety as in Section~\ref{sec:toric:variety}
and denote by $A$ the set of faces of the polytope $P$.

Let $X_\sigma$, $\sigma\in A$, be the collection of $T$-orbit closures in $X_P$ 
and let  $f_\sigma$, $\sigma\in A$, be a collection of $\hat T$-eigenfunctions $f_\sigma\in \mathbb K[\hat X_P]$
of degree $\deg f_\sigma\ge 1$. Denote by $\mu_\sigma$ the $\hat T$-weight of $f_\sigma$.
The following are equivalent:
\begin{itemize}
\item The collection $(X_\sigma,f_\sigma)_{\sigma\in A}$ of subvarieties and homogeneous functions defines a  
Seshadri stratification on $X_P$ which is $T$-equivariant in the sense of Definition~\ref{defn:T:equivariant1}.
\item The collection $(X_\sigma,f_\sigma)_{\sigma\in A}$ of subvarieties and homogeneous functions defines a  
combinatorial Seshadri stratification in the sense of Definition~\ref{defn:T_equi:SS}.
\end{itemize}
\end{theorem}

The proof of Theorem~\ref{Seshadri:is:Seshadri:toric:theorem1}  is divided into several steps. We start by proving:
\begin{lemma}\label{condition:one:lemma1}
The collection $X_\sigma$, $\sigma\in A$, of $T$-orbit closures in $X_P$ satisfies the condition (S1) for a Seshadri stratification.
\end{lemma}
\begin{proof}
Since $P$ is a normal polytope, the  
variety $X_P$ is a normal toric variety. And, by the general theory of toric varieties, 
so are the orbit closures $X_\sigma=\overline{O_\sigma}$ for $\sigma\in A$. 
In particular, the varieties $X_{\sigma}$ are smooth in codimension one.
The condition on the cover relations is satisfied by the fact that in the case of toric varieties, the complement
of an orbit $O_\sigma$ in its closure is the union of the orbit closures of the orbits of codimension 1 in $X_\sigma$.
\end{proof}
In the following, we assume always that the collection of subvarieties $X_\sigma$, $\sigma\in A$, is given
by the orbit closures.
\begin{lemma}\label{S3:equivalen:tS3strich1}
A collection of homogeneous $T$-eigenfunctions $f_\sigma\in \mathbb K[\hat X_P]$, $\sigma\in A$,
$\deg f_\sigma\ge 1$, has property \eqref{weight:condition11} in Definition~\ref{defn:T_equi:SS} 
if and only if it satisfies the condition (S3).
\end{lemma}
\begin{proof}
Let $\sigma$ be a face of $P$. If (S3) is satified by the collection of functions, then $f_\sigma$, $\sigma\in A$, 
does not vanish on $X_\sigma$, but $f_\sigma$ vanishes on $X_\tau$ for $\tau$ a proper face of $\sigma$.
A homogeneous $T$-eigenfunction in $ \mathbb K[\hat X_P]$ can be written (up to a non-zero scalar factor) as the restriction of a monomial in the 
$x_{\chi}, \chi\in \Lambda$. Now a coordinate function $x_{\chi}$ vanishes on 
$X_\sigma$ unless $\chi\in \sigma$. So $f_\sigma$ can be written as the restriction of a monomial in the $x_\chi,\chi\in\Lambda_\sigma$, 
and hence the weight $\mu_\sigma$ of $f_\sigma$ has the property: $-\mu_\sigma/\deg f_\sigma$ is an affine convex combinations of the
$\chi\in  \Lambda_\sigma$. In particular:
$\frac{-\mu_\sigma}{\deg f_\sigma}\in \sigma$. 
A face $\sigma$ is the disjoint union of the relative interiors
its faces.  So let $\tau\le \sigma$ be the unique face such that $\frac{-\mu_\sigma}{\deg f_\sigma}$ is in the relative interior
of $\tau$. If $\tau\not=\sigma$, then $f_\sigma$ must be a monomial in the $x_\chi, \chi\in  \Lambda_\tau$, and hence $f_\sigma$ is not identically
zero on $X_\tau$. So (S3) implies $\tau=\sigma$, and hence (S3) implies: 
$\frac{-\mu_\sigma}{\deg f_\sigma}\in \sigma^o$.

Vice versa, if  $\frac{-\mu_\sigma}{\deg f_\sigma}\in \sigma^o$ is an element in the relative interior of $\sigma$,
then the proof of Lemma~\ref{vanishing:set} implies (S3).
\end{proof}
\begin{proof}[Proof of Theorem~\ref{Seshadri:is:Seshadri:toric:theorem1}]
If the collection of subvarieties $X_\sigma$ and functions $f_\sigma$, $\sigma\in A$,
defines Seshadri stratification which is  $T$-equivariant  in the sense of Definition~\ref{defn:T:equivariant1}, 
then the collection of functions $f_\sigma$, $\sigma\in A$,
satisfies by Lemma~\ref{S3:equivalen:tS3strich1} also the condition \eqref{weight:condition11}.
Hence it is also a  combinatorial Seshadri stratification in the sense of Definition~\ref{defn:T_equi:SS}.

Vice versa, suppose  the collection of subvarieties $X_\sigma$ and functions $f_\sigma$, $\sigma\in A$, defines
combinatorial Seshadri stratification in the sense of Definition~\ref{defn:T_equi:SS}. So the subvarieties
are given by the $T$-orbit closures in $X_P$ and the extremal functions are $\hat T$-eigenfunctions,
hence the conditions (E1) and (E2) are satisfied.
The conditions (S1) and (S3) for a Seshadri stratification are automatically satisfied
by Lemma~\ref{condition:one:lemma1} and Lemma~\ref{S3:equivalen:tS3strich1}, and (S2) follows
by Lemma~\ref{vanishing:set}. So the collection $(X_\sigma,f_\sigma)_{\sigma\in A}$ of subvarieties and 
homogeneous functions defines a  Seshadri stratification on $X_P$ which is $T$-equivariant.
\end{proof}
\subsubsection{The valuations}
In \cite{CFL}, we use a Seshadri stratifications to define for every maximal chain $\mathfrak C$ in $A$
a valuation $\mathcal V_{\mathfrak C}:\mathbb K[\hat X_P]\setminus\{0\}\rightarrow \mathbb Q^\mathfrak C$,
using renormalized successive vanishing multiplicities. Before showing 
that the valuation $\nu_{\mathfrak C}$ defined in Definition~\ref{def:valuation} is equal
to the one defined in \cite{CFL}, we recall quickly the construction of $\mathcal V_{\mathfrak C}$ and some of its properties.
 
We add to the set $A$ the element $\{0\}$, i.e. $\hat A=A\cup\{0\}$. The variety $\hat X_0=\{0\}$
is just the origin in $V$ and hence contained in all the affine varieties $\hat X_\tau$, $\tau\in A$. 
So it makes sense to extend the partial order from $A$ to $\hat A$ by: $\tau> 0$ for all $\tau\in A$.

\subsubsection{The one-step case}\label{one:step}
Let $\tau>\sigma$ be a covering in $\hat A$. Since $\hat X_\tau$ is smooth in codimension one, we have a well defined
valuation $\mathcal V_{\tau,\sigma}: \mathbb K(\hat X_\tau)\setminus\{0\}\rightarrow \mathbb Z$, which associates to $g\in \mathbb K(\hat X_\sigma)\setminus\{0\}$
its vanishing multiplicity on $\hat X_\sigma$.

For $\tau\in A$ let $f_\tau$ be the fixed extremal function associated to $\tau$ and denote by $b_{\tau,\sigma}=\mathcal V_{\tau,\sigma}(f_\tau)$ 
the vanishing multiplicity of $f_\tau\vert_{X_\tau}$ on $\hat X_\sigma$. We associate to $g\in \mathbb K(\hat X_\tau)$ a new rational function on $\hat X_\sigma$
as follows: set
\begin{equation}\label{g:prime}
g':=\frac{g^{b_{\tau,\sigma}}}{f_\tau^{\mathcal V_{\tau,\sigma}(g)}\vert_{X_\tau}}.
\end{equation}
By construction, $g'$ is a well defined rational function on $\hat X_\tau$. 
It has been shown in \cite{CFL} (in a more general context):
\begin{lemma}[\cite{CFL}, Lemma 4.1]\label{dividing:procedure11}
The restriction $g'\vert_{\hat X_\sigma}$ is a well-defined, non-zero rational function on $\hat X_\sigma$ .
\end{lemma}

Suppose now $g\in \mathbb K(\hat X_\tau)$ is a $\hat T$-eigenfunction of weight $\lambda_g$. Recall that $f_\tau$
is a $\hat T$-eigenfunction, denote by $\mu_\tau$ its character. As a quotient of $\hat T$-eigenfunctions, the function
$g'$ itself is a $\hat T$-eigenfunction. By construction we see:
\begin{lemma}\label{weight:lemma}
If $g\in \mathbb K(\hat X_\tau)$ is a $\hat T$-eigenfunction of character $\lambda_g$, then
$g'\vert_{\hat X_\sigma}\in \mathbb K(\hat X_\sigma)$ is a $\hat T$-eigenfunction of character 
$\lambda_{g'}=b_{\tau,\sigma}\lambda_g - \mathcal V_{\tau,\sigma}(g)\mu_\tau$.
\end{lemma}
\subsubsection{The valuation associated to a maximal chain $\mathfrak C\subseteq A$}\label{valuation:maximal:chain1}
Let $\mathfrak C:\tau_r=P>\tau_{r-1}>\ldots>\tau_0$ be a maximal chain in $A$. We endow $\mathbb Q^\mathfrak C$ with the associated
lexicographic order and define a $\mathbb Q^\mathfrak C$-valued valuation
on $\mathbb K(\hat X_P)$ as follows:

To simplify the notation, we write just $\mathcal V_i$ instead of $\mathcal V_{\tau_i,\tau_{i-1}}$
for the valuation associated  to the cover $\tau_i>\tau_{i-1}$ in $A$. 
The element $\tau_0$ is a minimal element in $A$, we write $\nu_0$ for the valuation
given by the vanishing multiplicity of a rational function $g\in\mathbb K(\mathbb A^1)\setminus\{0\}$ in the origin.
We simplify in the same way the notation for the vanishing multiplicity $b_{\tau_i,\tau_{i-1}}$ of $f_{\tau_i}\vert_{\hat X_{\tau_i}}$
on $\hat X_{\tau_{i-1}}$, we write just $b_i$ instead.

We associate to a rational function $g\in \mathbb K(\hat X_P)$ a sequence of rational functions on the subvarieties corresponding to the elements 
in the fixed maximal chain: $g_{\mathfrak C}=(g_r,\ldots,g_1,g_0)$, where $g_r=g$,
and then we repeat the procedure in Lemma~\ref{dividing:procedure11}:
$g_{r-1}=g_r'\vert_{X_{\tau_{r-1}}}$, $g_{r-2}=g_{r-1}'\vert_{X_{\tau_{r-2}}}$, $g_{r-3}=g_{r-2}'\vert_{X_{\tau_{r-3}}}$ and so on.

\begin{definition}
Let $\{e_{\tau_i}\mid i=0,\ldots,r\}$ be the standard basis for $\mathbb Q^\mathfrak C$. We set:
\begin{equation}\label{valuation:C}
\mathcal V_{\mathfrak C}:  \mathbb K(\hat X_P)\setminus\{0\} \rightarrow \mathbb Q^\mathfrak C,\quad
g\mapsto \frac{\mathcal V_r(g_r)}{b_r}e_{\tau_r} + \frac{\mathcal V_{r-1}(g_{r-1})}{b_rb_{r-1}}e_{\tau_{r-1}} 
+\ldots + \frac{\mathcal V_{0}(g_{0})}{b_r\cdots b_{0}}e_{\tau_{0}}.
\end{equation}
\end{definition}
It has been proved in \cite{CFL}:
\begin{proposition}[\cite{CFL}, Proposition 6.10]
The map $\mathcal V_{\mathfrak C}$ is a $\mathbb Q^\mathfrak C$-valued valuation.
\end{proposition}

\begin{remark}\label{iterativ:formula:111}
Using an inductive procedure, one gets the following formula for $j=0,\ldots, r-1$: {\small
$$
g_{j}=g^{b_{j+1}}_{j+1}f^{-\nu_{j+1}(g_{j+1})}_{\tau_{j+1}}\vert_{\hat X_{\tau_{j}}}
=\ldots=
g^{b_r\cdots b_{j+1}}f_{\tau_r}^{-b_{r-1}\cdots b_{j+1}\nu_r(g_r)}\cdots f^{-b_{j+1}\nu_{j+2}(g_{j+2})}_{\tau_{j+2}} 
f^{-\nu_{j+1}(g_{j+1})}_{\tau_{j+1}}\vert_{\hat X_{\tau_{j}}}.
$$}
\end{remark} 

Our aim is to show that the two valuations $\mathcal V_{\mathfrak C}$ and $\nu_{\mathfrak C}$ coincide. 
The following Lemma is a first step:
\begin{lemma}\label{T-eigenfunction:equality}
If $g\in \mathbb K(\hat X_\tau)$ is a $\hat T$-eigenfunction of character $\lambda_g$ and $\mathcal V_{\mathfrak C}(g)=(a_r,\ldots,a_0)$, 
then $\lambda_g=a_r\mu_{\sigma_r}+\ldots+a_0\mu_{\sigma_0}$. In particular, $\mathcal V_{\mathfrak C}(g)=\nu_{\mathfrak C}(g)$.
\end{lemma}
\begin{proof}
We know by Lemma~\ref{weight:lemma}  that $g=g_r$, $\ldots$, $g_0$ are $\hat T$-eigenfunctions, and the corresponding characters
can be calculated by the formula $\lambda_{g_{j-1}}=b_{j}\lambda_{g_j} - \mathcal V_{j}(g_j)\mu_{\tau_j}$. So inductively we get:
$$
\lambda_g=\lambda_{g_r}=  \frac{\mathcal V_{r}(g_r)}{b_r}\mu_{\tau_r} + \frac{\lambda_{g_{r-1}}}{b_r}
=\ldots = a_r\mu_{\tau_r}+\ldots + a_0\mu_0,
$$
which finishes the proof.
\end{proof}
The next step is to reduce the problem to the case of $\hat T$-eigenfunctions.

\begin{lemma}\label{T-invariance:lemma:11}
\begin{itemize}
\item[{\it i)}] $\mathcal V_\mathfrak C$ is $\hat T$-invariant, i.e. $\mathcal V_\mathfrak C(g)=\mathcal V_\mathfrak C(t\cdot g)$
for $g\in \mathbb K(\hat X_P)\setminus\{0\}$ and $t\in \hat T$.
\item[{\it ii)}] For $g\in \mathbb{K}[\hat X_P]\setminus\{0\}$ let $g=g_{\eta_1}+\ldots + g_{\eta_q}$ be a decomposition of $g$ as a linear combination 
of $\hat T$-eigenfunctions $g_{\eta_i}$, $\eta_i\not=\eta_j$ for all $i\not=j$, $\eta_i\in \hat M$, $i=1,\ldots,q$. 
Then $\mathcal V_\mathfrak C(g)=\min\{\mathcal V_\mathfrak C(g_{\eta_j})\mid j=1,\ldots,q\}$.
\end{itemize}
\end{lemma}
Lemma~\ref{T-eigenfunction:equality}, the second part of Lemma~\ref{T-invariance:lemma:11} together with Definition~\ref{def:valuation} implies:
\begin{corollary}\label{max:chain:valuation:equal}
For all $g\in \mathbb K[\hat X_P]\setminus\{0\}$ and $\mathfrak C\in \mathcal F_{\max}(A)$ one has:
$\mathcal V_\mathfrak C(g)=\nu_\mathfrak C(g).$
\end{corollary}

\begin{proof}[Proof of  Lemma~\ref{T-invariance:lemma:11}] 
The action of $\hat T$ stabilizes the divisor $\hat X_{\tau_{j-1}}$ in $\hat X_{\tau_j}$ for $j=1,\ldots,r$. The associated
algebra automorphisms stabilizes hence the associated maximal ideal in the local ring 
$\mathcal O_{\hat X_{\tau_j},\hat X_{\tau_{j-1}}}\subset \mathbb K(\hat X_{\tau_{j}})$ for
$j=1,\ldots,r$. So for all $g\in \mathbb K(\hat X_{{\sigma_j}})$: the vanishing multiplicities of $g$ respectively $t\cdot g$
on $\hat X_{\sigma_{j-1}}$ are the same.

To prove \textit{i)}, 
for $g\in \mathbb K[\hat X_P]$ let $g_\mathfrak C=(g_r,\ldots,g_0)$ be the associated sequence of rational functions.
We can use the $\hat T$-action to construct two new tuples: for $t\in \hat T$ consider the $t$-twisted
tuple $(t\cdot g_r,\ldots,t\cdot  g_0)$ obtained by twisting component-wise each of the rational functions 
in the sequence associated to $g$. And we have the sequence 
$(g'_r,\ldots,g'_0)$ associated to the function $t\cdot g$.

The functions $\{f_{\sigma}\mid \sigma\in A\}$ are $T$-eigenfunctions, so by Remark~\ref{iterativ:formula:111}
the $j$-th component $t\cdot g_j$ in the $t$-twisted sequence 
and the $j$-th component $g'_j$ in  the sequence associated to $t\cdot g$ differ only by a nonzero scalar multiple.
It follows: $\mathcal V_{\mathfrak C}(t\cdot  g)=\mathcal V_{\mathfrak C}(g)$ for
$g \in \mathbb{K}[\hat X_P]\setminus\{0\}$.

To prove \textit{ii)}, fix $g \in \mathbb{K}[\hat X_P]\setminus\{0\}$ and let $g=g_{\eta_1}+\ldots + g_{\eta_q}$  be a decomposition
into $\hat T$-eigenfunctions, we suppose the characters are pairwise different.
We know: $\mathcal V_{\mathfrak C}(g)\ge \min\{\mathcal V_{\mathfrak C}(g_{\eta_j})\mid 1\leq j\leq q \}$
by the minimum property of a valuation. The assumption on the characters to be pairwise different
implies that one find $t\in\hat T$ such that $t\cdot g_{\eta_j}=c_jg_{\eta_j}$ for  
pairwise distinct $c_1,\ldots,c_q\in\mathbb{K}^*$. It follows that 
the linear span of the following functions coincide:
$$
\langle g_{\eta_1},\ldots ,g_{\eta_q}\rangle_{\mathbb K} =\langle g,t\cdot g,\ldots,t^{q-1}\cdot g\rangle_{\mathbb K}.
$$
So one can express the $\hat T$-eigenfunction $g_{\eta_j}$ as a linear combination of  
$g, t\cdot g, \ldots, t^{q-1}\cdot g$. 
Now  part \textit{i)} of Lemma~\ref{T-invariance:lemma:11}
implies: for all $j=1,\ldots,q$,
$$
\mathcal V_{\mathfrak C}(g_{\eta_j})\ge  \min\{\mathcal V_{\mathfrak C}({t^i}\cdot g)\mid i=0,\ldots, q-1\}
=\mathcal V_{\mathfrak C}(g),
$$
 and hence $\mathcal V_{\mathfrak C}(g)=\min\{\mathcal V_{\mathfrak C}(g_{\eta_j})\mid j=1,\ldots,q\}$.
\end{proof}
\subsection{The quasi-valuations}
In \cite{CFL} we have used the valuations $\mathcal V_{\mathfrak C}$ to define a quasi-valuation. Rephrased 
in terms of an embedded toric variety $X_P\hookrightarrow \mathbb P(V)$ the definition in \cite{CFL} reads as:
\begin{definition}\label{def:quasivaluation2}
The \emph{quasi-valuation $\mathcal V:\mathbb{K}[\hat{X_P}]\setminus\{0\}\to\mathbb{Q}^A$ }
associated to the $T$-equivariant Seshadri stratification $(X_\sigma,f_\sigma)_{\sigma\in A}$ and the fixed 
total order $>^t$ on $A$ is the map  defined by:
$$
g\mapsto\nu(g):=\min\{\mathcal V_{\mathfrak{C}}(g)\mid \mathfrak{C}\in \mathcal F_{\max}(A)\}.
$$
\end{definition}

\begin{corollary}
Let  $\nu$ be the quasi-valuation on $ \mathbb K[\hat X_P]$ defined in Definition~\ref{def:quasivaluation}
and let $\mathcal V$ be the quasi-valuation  $\mathbb K[\hat X_P]$ defined in \cite{CFL}.
For all $g\in \mathbb K[\hat X_P]\setminus\{0\}$ one has:
$\mathcal V(g)=\nu(g)$.
\end{corollary}
\begin{proof}
For $g\in \mathbb K[\hat X_P]\setminus\{0\}$ one has by definition and by Corollary~\ref{max:chain:valuation:equal}:
$$
\mathcal V(g)=\min\{\mathcal V_{\mathfrak{C}}(g)\mid \mathfrak{C}\in \mathcal F_{\max}(A)\}
=\min\{\nu_{\mathfrak{C}}(g)\mid \mathfrak{C}\in \mathcal F_{\max}(A)\}=\nu(g).
$$
\end{proof}


\begin{thebibliography}{99}







\bibitem{CFL} 
R. Chiriv\`i, X. Fang, P. Littelmann, 
\emph{Seshadri stratifications and standard monomial theory}, 
Invent. Math. {\bf 234}, 489--572, (2023).

\bibitem{CFL3}
R. Chiriv\`i, X. Fang, P. Littelmann, 
\emph{Seshadri stratification for Schubert varieties and standard Monomial Theory}, 
Proc. Indian Acad. Sci. Math. Sci. 132, no. 2, Paper No. 74, 58 pp., (2022).

\bibitem{CFL4}
R. Chirivì, X. Fang, P. Littelmann: \emph{Seshadri stratifications and Schubert varieties: a geometric construction
of a standard monomial theory}. Pure Appl.Math. Q. Volume 20, Number 1, 139--169, (2024). 



\bibitem{CLS}
D. Cox, J. Little, H. Schenck, 
\emph{Toric Varieties}. Graduate Studies in Mathematics. vol. {\bf 124}. American Mathematical Society, Providence (2011).

\bibitem{E} 
D. Eisenbud, 
{Commutative algebra with a View Toward Algebraic Geometry}, 
Springer Verlag (1995). 


\bibitem{FL}
X. Fang, P. Littelmann, \emph{From standard monomial theory to semi-toric degenerations via Newton-Okounkov bodies}, 
Trans. Moscow Math. Soc. {\bf 78}, 275--297,  (2017).

\bibitem{GKZ}
I. Gelfand, M. Kapranov, A. Zelevinsky, \emph{Discriminants, Resultants and Multidimensional
Determinants}, Birkh\"auser, Boston, (1994). 



\bibitem{Ho}
T. Hosgood, 
\emph{An introduction to varieties in weighted projective space},
arXiv:1604.02441v5 [math.AG] (2020).
\bibitem{KTv}
G. Kemper, N.V. Trung and N. T. van Anh, 
\emph{Toward a theory of monomial preorders}, 
Math. Comp. {\bf 87}, no. 313, 2513--2537, (2018).

\bibitem{KR}
M. Kreuzer and L. Robbiano, 
\emph{Computational Commutative Algebra 2}, Springer-Verlag,
Berlin, (2005.) 

\bibitem{Zhu}
C.-G. Zhu,
\emph{Degenerations of toric ideals and toric varieties}
J. Math. Anal. Appl. {\bf 386}, 613--618,  (2012).





\end{thebibliography}
\end{document}